\documentclass[12pt]{amsart}

\usepackage{amscd,latexsym,amsthm,amsfonts,amssymb,amsmath,amsxtra}
\usepackage[mathscr]{eucal}
\usepackage{hyperref}
\usepackage{stmaryrd}
\usepackage{MnSymbol} 
\renewcommand\theequation{\thesection.\arabic{equation}}

\newcommand{\BA}{{\mathbb {A}}}

\newcommand{\RD}{{\mathrm {D}}}
\newcommand{\RE}{{\mathrm {E}}}

\newcommand{\RG}{{\mathrm {G}}}

\newcommand{\Aut}{{\mathrm{Aut}}}

\newcommand{\End}{{\mathrm{End}}}

\newcommand{\FJ}{{\mathrm{FJ}}}

\newcommand{\GL}{{\mathrm{GL}}}

\newcommand{\Hom}{{\mathrm{Hom}}}

\newcommand{\SL}{{\mathrm{SL}}}
\newcommand{\Spin}{{\mathrm{Spin}}}

\newcommand{\SO}{{\mathrm{SO}}}

\newcommand{\Sym}{{\mathrm{Sym}}}

\newcommand{\Sp}{{\mathrm{Sp}}}

\newcommand{\wt}{\widetilde}

\newcommand{\ol}{\overline}
\newcommand{\ul}{\underline}

\newcommand{\bs}{\backslash}

\newtheorem{thm}{Theorem}[section]
\newtheorem{cor}[thm]{Corollary}
\newtheorem{lem}[thm]{Lemma}
\newtheorem{prop}[thm]{Proposition}

\newtheorem {ques/conj}[thm]{Question/Conjecture}

\newtheorem{defn}[thm]{Definition}

\makeatletter

\newcommand{\Rmnum}[1]{\expandafter\@slowromancap\romannumeral #1@}
\makeatother

\begin{document}
\renewcommand{\theequation}{\arabic{equation}}
\numberwithin{equation}{section}

\title[Raising Nilpotent Orbits]{Raising  Nilpotent Orbits in Wave-Front Sets} 
\author{Dihua Jiang}
\address{School of Mathematics\\
University of Minnesota\\
Minneapolis, MN 55455, USA}
\email{dhjiang@math.umn.edu}

\author{Baiying Liu}
\address{Department of Mathematics\\
University of Utah\\
Salt Lake City, UT, 84112, USA}
\email{liu@math.utah.edu}

\author{Gordan Savin}
\address{Department of Mathematics\\
University of Utah\\
Salt Lake City, UT, 84112, USA}
\email{savin@math.utah.edu}

\date{\today}



\begin{abstract}
We study wave-front sets of representations of reductive groups over global or non-archimedean local fields.
\end{abstract}

\maketitle


\section{Introduction} 

Let $k$ be a global or non-archimedean local field. Let $G(k)$, or simply $G$, 
 be the group of $k$-points of a reductive algebraic group defined over $k$, 
or a central extension of finite degree. 
Let $\frak{g}$ be the Lie algebra of $G$ and $u$ a nilpotent element in $\frak{g}$. We assume that the characteristic of $k$ is large, so that 
the Jacobson-Morozov theorem holds, i.e. there exists a homomorphism $\varphi : \frak{sl}_2 \rightarrow \frak{g}$ such that 
\[ 
u=\varphi \left(\begin{matrix}
0 & 0\\
1 & 0
\end{matrix}\right). 
\] 

Now assume that $k$ is a local field. Fix a non-trivial character $\psi$ of $k$. There is a 
unipotent subgroup  $N_u\subseteq G$ and a character $\psi_u$ of $N_u$, corresponding to $u$ (see Section 5 for definitions). For example, if $u$ belongs to a regular orbit, then $N_u$ is a maximal 
unipotent subgroup and $\psi_u$ is a Whittaker character. Let $\pi$ be a smooth representation of $G$, not necessarily admissible. 
The {\it wave-front set} of $\pi$ is the set of nilpotent orbits 
$\mathcal O$ such that the space of twisted co-invariants $\pi_{N_u, \psi_u}$ is nontrivial, for $u\in \mathcal O$.  Our main result, Corollary \ref{cor:local},  concerns the structure of orbits in the wave-front set.  More precisely, let  $\mathcal O$ be contained in the wave-front set of $\pi$. Then, under certain conditions, a slightly larger orbit $\mathcal O'$  is also contained in the wave-front set of $\pi$. 

Assume, for example, that $G$ is a classical group and $\mathcal O$ corresponds to 
a partition $\ul{p}$. Then the conditions are automatically satisfied if $\ul{p}$ is not special, in the sense of Lusztig and Spaltenstein. The larger orbit  
$\mathcal O'$ corresponds to a partition $\ul{p}'$  obtained from $\ul{p}$ by replacing a pair $(i,i)$ in $\ul{p}$ by $(i-1,i+1)$. This process can be continued until we arrive to a nilpotent element whose corresponding partition is special, more precisely, the special expansion $\ul{p}^G$ of $\ul{p}$. 
 In particular, the maximal orbits in the wave-front set,  with respect to the closure ordering, are special. This result, if $k$ has characteristic zero, was previously obtained by M\oe glin  in \cite{Mo96} for irreducible representations using the relationship of degenerate Whittaker models and the character expansion of (irreducible) representations obtained in the work of M\oe glin and Waldspurger in \cite{MW87}.

We also have analogous results if $k$ is a global field. Let $\mathbb A$ be the corresponding ring of ad\` eles. Let $\psi$ be a non-trivial character 
of $\mathbb A$, trivial on $k$. Let $\Pi$ be a space of smooth functions on $G(k)\backslash G(\mathbb A)$ stable under the action of $G(\mathbb A)$ by right translations. Every nilpotent element $u$  in $\frak{g}$ defines a character $\psi_u$ of $N_u(k)\backslash N_u(\mathbb A)$ and functional on $\Pi$ by 
\[ 
\int_{N_u(k)\backslash N_u(\mathbb A)} f(n) \bar{\psi}_u(n) dn 
\] 
where $f\in \Pi$.  The set of nilpotent $G(k)$-orbits such that this functional is non-trivial is a {\it (global) wave-front set} of $\Pi$. 
 Let  $\mathcal O$ be an orbit  contained in the wave-front set of $\Pi$. 
 Corollary \ref{cor:global},  proved by an argument analogous to the one in the local setting, states that, under certain conditions, a slightly larger orbit $\mathcal O'$  is also contained in the wave-front set of $\Pi$. As a consequence, if $G$ is classical, only special orbits are maximal in the wave-front set of $\Pi$. 
For a background of the global result the reader can consult  \cite{GRS03}, \cite{G06} or \cite{JL15}. In fact, the basic idea of this paper is already contained in 
these papers, for example, in Theorem 3.1 of \cite{G06}.  

 Finally, we consider split exceptional groups. In this case there are 
several non-special orbits that cannot be eliminated, as maximal elements in the wave-front set,  by either our or M\oe glin's  method.  
These orbits are completely odd, non-special orbits, as conjectured by Nevins \cite{N02}, and are denoted by 
$\ast\ast$ in the tables given in Sections \ref{S:G_2} - \ref{S:E_8}. 
The first example of such orbit is the minimal orbit for the exceptional group $\RG_2$. However, this orbit cannot be a maximal element in the wave front set by a more 
elaborate argument contained in \cite{LS}. Thus, it is still reasonable to expect that, for algebraic groups,  only special orbits appear as the maximal elements in the wave-front set of a representation of $G$.

\section{Heisenberg group} 

Assume, from now on, that the characteristic of $k$ is not $2$. Let $\frak{h}$ be a Heisenberg Lie algebra over $k$ with the center $\frak{z}$. 
Let $H=\exp(\frak{h})$ be the Heisenberg group. As a set, $H=\frak{z}$, but the multiplication 
in $H$ is given by the Campbell-Hausdorff formula,
\[ 
x\cdot y=x+y + \frac{1}{2}[x,y]. 
\] 
The center of $H$ is $Z=\exp(\frak{z})$. 
A polarization of $H$ is a decomposition of the Lie algebra 
\begin{equation}\label{eq2} 
\frak{h} = \log(X)\oplus \log (Y) \oplus \frak{z} 
\end{equation}
such that $X=\exp(\log X)$ and $Y=\exp(\log Y)$ are abelian subgroups of $H$. 
 The group $XZ$ is a maximal abelian subgroup of $H$. 
 
Assume now that $k$ is a local field. Fix a non-trivial character $\psi$ of $k$. By choosing an identification $Z$ with $k$ we view $\psi$ as a character of $Z$. 
For every $y\in Y$  we have a character 
$\psi_y: X\rightarrow \mathbb C^{\times}$ given by  
\[ 
\psi_{y}(x)= \psi([y,x]) \text{ for every } x\in X 
\] 
where $[x,y]=xyx^{-1}y^{-1}$.   Let $\rho_{\psi}$ be the representation of $H$ obtained by the smooth induction of the character $1\boxtimes \psi$ of $XZ$. 
The representation $\rho_{\psi}$ is realized on the space $S(Y)$ of Schwartz functions on $Y$. The action on $f\in S(Y)$ is given by 
\[ 
\rho_{\psi}(y) (f)(u) = f(uy)\text{ for every } y\in Y
\] 
  and 
\[ 
  \rho_{\psi}(x)(f)(u) = \psi_u(x) f(u) \text{ for all } x\in X. 
\] 
The same formulae also define the action of $H$ on $L^2(Y)$. This is the Schr\" odinger model of the unique  irreducible unitary representation of $H$ with the central character $\psi$. The subspace of $H$-smooth vectors in $L^2(Y)$ is $S(Y)$.

\begin{prop} \label{P1:local} 
 Let $k$ be a non-archimedean field. Let $\pi$ be a smooth representation of $H$ such that $Z$ acts on $\pi$ as the character $\psi$. The bilinear map 
 $\Hom_{H}(\rho_{\psi}, \pi)\times \rho_{\psi} \rightarrow \pi$ given by $(A,v) \mapsto A(v)$ descends to a canonical isomorphism 
\[ 
\Hom_{H}(\rho_{\psi}, \pi)\otimes \rho_{\psi} \cong \pi. 
\] 
If $\Hom_{H}(\rho_{\psi}, \pi)\neq 0$ then  $\pi_{X, \psi_y}\neq 0$, for every $y\in Y$. 
\end{prop}  
\begin{proof} The first part, the isomorphism given by $A\otimes v \mapsto A(v)$, is in \cite{W03}. 
If $\ell$ is an arbitrary functional on $\Hom_{H}(\rho_{\psi}, \pi)$ and $\ell_y$ the functional on 
$S(Y)$ given by evaluating functions $f\in S(Y)$ at $y$, then $\ell\otimes \ell_y$ transforms under the action of $X$ as $\psi_y$. This proves the second part. 
\end{proof} 

Now assume that $k$ is a global field.  Let $k_v$ denote either a local non-archimedian completion of $k$, or $k_{\infty}=k\otimes_{\mathbb Q} \mathbb R$. 
Let $\mathbb{A}$ be the corresponding ring of ad\` eles.  It is a restricted product of all $k_v$. 
 Let $\psi$ be an additive character of $\mathbb A$ trivial on $k$. 
As in the local case, every $y\in Y(\mathbb A)$,  defines a character $\psi_y$ of $X(\mathbb{A})$. If $y\in Y(k)$ the the character $\psi_y$ is trivial on $X(k)$. 
The group $H(\mathbb A)$ has  an irreducible
unitary representation  with the central character $\psi$.  This representation is unique up to isomorphism. 
It is realized on $L^2(Y(\mathbb A))$.  
The subspace of $H(\mathbb A)$-smooth vectors in this realization
 is $S(Y(\mathbb A))$. This space is isomorphic to the restricted tenor product of $S(Y(k_v))$. 
 
 The unique irreducible representation of $H(\mathbb A)$  has another realization on 
$L^2_{\psi}(H(k)\backslash H(\mathbb{A}))$, where the subscript $\psi$ denotes the subspace  of 
functions  transforming as $\psi$ under the action of $Z(\mathbb{A})$. 
 This is the lattice model. In this realization, 
the subspace  $\Theta_{\psi}$ of $H(\mathbb A)$-smooth vectors is the space of smooth functions on $H(k)\backslash H(\mathbb{A})$ 
transforming as $\psi$ under the action of $Z(\mathbb{A})$. 
 For every $\theta \in \Theta_{\psi}$ and $u\in Y(\mathbb A)$, let 
\[ 
f_{\theta}(u)= 
\int_{X(k)\backslash X(\mathbb{A})} \theta(xu) dx. 
\] 
Then $\theta \mapsto f_{\theta}$ is an isomorphism of $\Theta_{\psi}$ and $S(Y(\mathbb A))$, the spaces of $H(\mathbb A)$-smooth vectors in the two models of the Heisenberg representation.

\begin{prop} \label{P1:global} 
Let $k$ be a global field. Let $\Pi$ be a non-zero subspace of  $\Theta_{\psi}$, stable under the action of $H(\mathbb A)$ by the right translations. 
   Then, for every $y\in Y(k)$,  there exists $\theta\in \Pi$ such that 
\[ 
\int_{X(k)\backslash X(\mathbb{A})} \theta(x) \bar{\psi}_y(x) dx \neq 0. 
\] 
\end{prop}  
\begin{proof} Since the map $\theta \mapsto f_{\theta}$ is injective, there exists $u'\in Y(\mathbb A)$ and $\theta'\in \Pi$ such that 
$f_{\theta'}(u')\neq 0$. Let $\theta$ be the right translate of $\theta'$ by  $y^{-1}u'$. Then 
\[ 
 \int_{X(k)\backslash X(\mathbb{A})} \theta(xy) dx \neq 0.
\]  
Next, we have the following  easy sequence of equalities, 
\[ 
\theta(xy)=\theta([x,y]yx)=\theta(yx) \bar{\psi}_y(x)=\theta(x)\ \bar{\psi}_y(x)
\] 
where, for example, the last identity holds since $\theta$ is left $H(k)$-invariant. Substituting $\theta(x)\ \bar{\psi}_y(x)$ for $\theta(xy)$ in the integral yields 
the proposition. 

\end{proof}

\section{Jacobi group} 

In this section we first introduce a type of Jacobi group that we shall need, and then define a notion of generic characters of certain abelian, 
unipotent subgroups of the Jacobi group. 

For our purposes, a Jacobi group $J$ is a semi-direct product of $M$, a central extension of ${\SL}_2$, and the Heisenberg group $H$.  
Let $\frak{sl}_2$ be the Lie algebra of $M$. We assume that, under the adjoint action of $\frak{sl}_2$, the Lie algebra $\frak{h}$ of $H$ decomposes as $m V_2 \oplus \frak{z}$, where $V_2$ is the irreducible 2-dimensional representation of $\frak{sl}_2$. 
 Let 
 \[ 
 s=\left(\begin{matrix} 1 & 0 \\ 0 & -1\end{matrix}\right). 
 \] 
 This assumption implies that there is a polarization (\ref{eq2}) of $\frak{h}$ 
where $\log X$ and $\log Y$ are the spaces of $s$-weight 1 and -1, respectively.  In particular, $\log X\oplus \frak{z}$ is a maximal abelian subalgebra of $\frak{h}$. 
Let $\frak{n}$ be the unipotent subgroup of upper triangular matrices in 
$\frak{sl}_2$, i.e. elements in $\frak{sl}_2$ of $s$-weight 2.  Then  
 \[ 
 \frak{u}=\frak{n} \oplus \log X \oplus\frak{z}
 \] 
  is an abelian sub-algebra of the Lie algebra of $J$.  Let $U=\exp(\mathfrak u)$. 
In order to discuss characters of $U$, it 
will be convenient to work with an explicit realization of $J$. 

Let $e_0, \ldots , e_{m}, f_{m}, \ldots , f_0$ be a basis of a vector space of dimension $2m+2$. 
Let $\langle,\rangle$ be a symplectic form such that 
\[ 
\langle e_i, e_j\rangle =0, \langle f_i, f_j\rangle =0\text{ and } \langle e_i, f_j\rangle =\delta_{ij}. 
\] 
Let $\Sp_{2m+2}$ be the group of linear transformations preserving the form $\langle,\rangle$, and $\frak{sp}_{2m+2}$ the corresponding Lie algebra. 
If $A$ is a square matrix, let $A^{\top}$ denote the transpose of $A$ with respect to the opposite diagonal.
The Lie algebra  $\frak{sp}_{2m+2}$,  in the basis $e_0, \ldots e_{m}, f_{m}, \ldots f_0$,   consists of block matrices 
\[ 
\left(\begin{matrix} A & B \\ 
C & D \end{matrix}\right) 
\] 
where $B=B^{\top}$, $C=C^{\top}$ and $D=-A^{\top}$. In this identification, the Killing form on the Lie algebra is given by the trace pairing.

In this setting the Heisenberg group $H$ is the unipotent radical of the maximal parabolic subgroup $P$ of $\Sp_{2m+2}$ stabilizing the line through $e_0$. 
 A Levi factor of $L$ of $P$ is given as the subgroup of all elements in $P$ stabilizing the line through $f_0$.  The derived group of $L$ is $\Sp_{2m}$, the subgroup 
 of $\Sp_{2m+2}$ fixing $e_0$ and $f_0$.  The conjugation action 
 of $\Sp_{2m}$ on $H$ gives an isomorphism of $\Sp_{2m}$ and the group of outer automorphisms of $H$. In order to write down $J$, we need to choose an embedding of 
 $\SL_2$ into $\Sp_{2m}$. Roughly speaking, we embed $\SL_2$ diagonally into $m$ long-root $\SL_2$'s. Up to conjugation, all such embeddings can be described as follows. 
 
  Fix an $m$-tuple $(a_1, \ldots , a_{m})$ of non-zero elements in $k^{\times}$. 
Let $\SL_2$ be the subgroup of $\Sp_{2m}$ such that,  for every $i\neq 0$,  it acts on the plane spanned by $e_i,f_i$ in the standard way with  respect to the basis $a_i e_i, f_i$.  We let $J$ be the semi direct product of $H$ with $M$, a central extension of this $\SL_2$. 

The algebra $\frak{u}=\frak{n}\oplus  \log X \oplus\frak{z}$ now 
consists of all matrices such that $A=C=D=0$ and (in the case $m=2$) 
\[ 
B=\left(\begin{matrix} x_2 & x_1 & z  \\ 
0 & n_1 & x_1 \\
n_2 & 0 & x_2 \end{matrix}\right) 
\] 
where $n_i=a_i n$, for some $n\in k$. 
Thus, we identify $\frak{u}$ with triples $(n,x,z)$ where $x=(x_1, \ldots , x_{m})$. Using the Killing form on $\frak{sp}_{2m+2}$ 
the dual space of $\frak{u}$ is identified with the quotient $\frak{sp}_{2m+2}/\frak{u}^{\perp}$. As a complement of 
$\frak{u}^{\perp}$ in  $\frak{sp}_{2m+2}$ we can take the set of all matrices such that $A=B=D=0$ and (in the case $m=2$) 
\[ 
C=\left(\begin{matrix} x^*_2 & 0 & n^*_2  \\ 
x^*_1 & n^*_1 & 0 \\
z^*& x^*_1 & x^*_2 \end{matrix}\right) 
\] 
where $n^*_i=a^{-1}_i n^*$, for some $n^*\in k$.
Thus, we identify the dual $\frak{u}^*$ with the set of triples $(n^*,x^*,z^*)$, where $x=(x^*_1, \ldots , x^*_{m})$. 
Now the natural pairing $\langle \cdot, \cdot \rangle$  between $\frak{u}$ and $\frak{u}^*$  is explicitly given by 
\[ 
\langle (n,x,z), (n^*,x^*,z^*) \rangle= Tr(BC)= zz^* +2\sum_{i=1}^{m}x_ix^*_i  +m nn^*. 
\] 

The commutative sub algebra $\log Y$ consists of matrices such that $B=C=0$ and (in the case $m=2$) 
\[ 
A=\left(\begin{matrix} 0 & y_2 & y_3  \\ 
0 & 0 & 0 \\
0 & 0 & 0 \end{matrix}\right). 
\] 
Thus, any element in $\log Y$ is identified with an $m$-tuple $y=(y_1, \ldots , y_m)$. The group $Y$ normalizes $U$. Hence $Y$ acts, by conjugation, on
 $\frak{u}$ and we have a co-adjoint action of $Y$ on $\frak{u}^*$. A short calculation shows that $\exp(y)$ acts on $(n^*,x^*, z^*)\in \frak{u}^*$ by 
\[ 
(n^*,x^*, z^*) \mapsto (n^* - 2x^*\cdot y +z^* y\cdot y , x^*-z^*y, z^*)
\] 
where the dot product  on $k^m$ is the one weighted by $a_i$'s, i.e. 
\[ 
y\cdot y=\sum_{i=1}^{m} a_i y_i y_i . 
\]  
It follows that the homogeneous polynomial $\Delta : \frak{u}^* \rightarrow k$ 
\[ 
\Delta(n^*,x^*,z^*)=n^*z^*-x^*\cdot x^*
\] 
is invariant under the action of $Y$. Note that $(n^* ,x^*, z^*)$ is in the $Y$-orbit of  $(\Delta/z^*, 0, z^*)$, if $z^*\neq 0$
.  We record this in the following proposition. 

\begin{prop} \label{P3:generic}  The polynomial $\Delta(n^*,x^*, z^*)=n^*z^* -x^*\cdot x^*$ on the linear space $\frak{u}^*$ 
  is invariant under the co-adjoint action of $Y$. Any element 
 $(n^* ,x^*, z^*)$ with $z^*\neq 0$  is the $Y$-orbit  of $(\Delta/z^*, 0,z^*)$.  
\end{prop}   

If $\psi$ is  a nontrivial character of $k$ then any $u^*\in \frak{u}$ defines a character $\psi_{u^*}$  of $U$ by 
\[ 
\psi_{u^*}(\exp(u)) =\psi(\langle u,u^*\rangle).  
\]  
\begin{defn} \label{D:generic} 
The character $\psi_{u^*}$  of $U$ is \underline{generic} if $\Delta(u^*)\neq 0$. 
\end{defn}

\section{Fourier-Jacobi models}

In this section we define local and global Fourier-Jacobi models of representations of the Jacobi group $J$ with non-trivial action of the center $Z$. The 
Fourier-Jacobi model is a representation of $M$.  If the Fourier-Jacobi model is Whittaker generic, then the original representation of $J$ is generic in the sense of 
Definition \ref{D:generic}. 

Fix a non-trivial character 
$\psi$ of $Z$. Realize $J$ as a  subgroup of $\Sp_{2m+2}(k)$, as in the previous section. 
 In particular, we have fixed an identification of $Z$ and $k$. In this way, $\psi$ can be viewed as 
a character of $k$.  Now every  $u^*\in \frak{u}$ gives a character $\psi_{u^*}$ of $U$. 


\smallskip

Assume now that $k$ is a local field.  Then, by Weil, the representation  $\rho_{\psi}$ extends to $J$. We shall need the following about this extension. 
The group of upper triangular matrices in $\SL_2(k)$ canonically splits in any central extension. In particular any element $n$ in $N=\exp(\mathfrak{n})\subset J$ is  
uniquely represented by a matrix $\left(\begin{smallmatrix}
1 & x\\
0 & 1
\end{smallmatrix}\right)$. Now for every $f \in S(Y)$ and $y=\exp(y_1, \ldots , y_{m}) \in Y$,
\[ 
(\rho_{\psi}(n) f)(y)= \psi(x \sum_{i=1}^{m} a_iy_i^2) f(y) ,
\] 
 where $a_i$'s are as in Section 3.

 Let $\pi$ be a smooth representation of $J$ such that the center $Z$ of $H$ acts on $\pi$ by $\psi$. 
 Then there is a representation $\sigma$ of $J$ on $FJ_{\psi}(\pi):= \Hom_{H}(\rho_{\psi}, \pi)$ defined by 
\[ 
A \mapsto \sigma(g)(A)= {\pi}(g)  \cdot A \cdot  \rho_{\psi}(g^{-1}).
\] 
for every $g\in J$ and $A\in \Hom_{H}(\rho_{\psi}, \pi)$. 
Now the isomorphism in Proposition \ref{P1:local} is an isomorphism of $J$-modules. 
Note that the subgroup $H$ acts \underline{trivially}. 
Hence $(\sigma, FJ_{\psi}(\pi))$ is a representation of $M$.

 \begin{prop}\label{P2:local} 
  Let $k$ be a non-archimedean field. 
  Let $\pi$ be a smooth $J$-module such that $Z$ acts on $\pi$ by the character $\psi$. 
If $(\sigma, FJ_{\psi}(\pi))$ is a Whittaker-generic representation of $M$,
then $\pi_{U, \psi_{u^*}}\neq 0$ for $u^*=(n^*,0,1)$ for some  $n^*\in k^{\times}$. 
\end{prop} 

\begin{proof} 
Since $\Hom_H(\rho_{\psi}, \pi)$ is Whittaker-generic, as a representation of $M$,  there is $n^* \in k^{\times}$ and a 
a non-zero functional $\ell$ on $\Hom_H(\rho_{\psi}, \pi)$  such that 
\[ 
\ell(\sigma(n) A) = \psi(n^*x) \ell(A) 
\] 
for all $A\in \Hom_{H}(\rho_{\psi}, \pi)$  and $n \in N$, where $n$ is represented by $\left(\begin{smallmatrix}
1 & x\\
0 & 1
\end{smallmatrix}\right)$. 
Recall that we have a decomposition 
\[ 
\pi= \sigma\otimes \rho_{\psi}, 
\] 
where  $\rho_{\psi}$ acts on $S(Y)$. Let 
\[ 
\ell_1: S(Y) \rightarrow \mathbb C 
\] 
be the functional given by evaluating functions at $1\in Y$.  Then $\ell\otimes \ell_1$ is a functional of $\pi$ that transforms as $\psi_{u^*}$ under the action of $U$, 
where $u^*=(n^*,0,1)$. 
\end{proof} 

Let $k$ be a global field, and fix a non-trivial character $\psi$ of $Z(\mathbb A)/Z(k)$. 
 Let $\rho_{\bar{\psi}}$ (note the complex conjugate!)  be the Weil representation of $J(\BA)$ on  $S(Y(\BA))$. 
For any  function $\phi \in S(Y(\BA))$ define a theta series 
\[
\theta^{\phi}_{\bar{\psi}}(hg)= \sum_{\xi \in Y(k)} \rho_{\bar\psi}(hg)\phi(\xi)= \sum_{\xi \in Y(k)} \rho_{\bar{\psi}}(\xi hg)\phi(1)
\] 
 where $h \in H(\BA)$ and $g \in M(\BA)$.

  Let $\Pi$ be a space of smooth functions on $J(k) \backslash J(\mathbb A)$ stable under the action of $J(\mathbb A)$ by right translations and such that $Z(\mathbb A)$ acts  by the  character $\psi$. For every $\phi \in S(\BA)$ and $f\in \Pi$, the function $h\mapsto f(hg) \theta_{\phi}(hg)$, where $h\in H(\BA)$, is left $Z(\BA)$-invariant, hence it can be viewed as a function on $V=H/Z$. 
 Let $\FJ_{\psi}(\Pi)$ be the representation of $M(\BA)$ spanned by the functions
 \[
 F_{f,\phi}(g) := \int_{V(k) \bs V(\BA)} f(vg)\theta^{\phi}_{\bar\psi}(vg)dv,
 \]
 where $f$ runs over $\Pi$, $\phi$ runs over $S(Y(\BA))$. Using the definition of $\theta^{\phi}_{\bar{\psi}}$ and that $f$ is left $\xi$-invariant for all 
 $\xi\in Y(k)$, the expression for $ F_{f,\phi}(g)$ can be made more explicit as
  \[
 F_{f,\phi}(g) = \int_{X(k) \bs V(\BA)} f(vg)[\rho_{\bar\psi}(vg) \phi](1)dv. 
 \]
 
 The following  is a global analogue of Proposition \ref{P2:local}. 

\begin{prop} \label{P2:global}
Let $k$ be a global field. 
 Let $\Pi$ be a space of smooth functions on $J(k) \backslash J(\mathbb A)$ stable under the action of $J(\mathbb A)$ by right translations and such that $Z(\mathbb A)$ acts  by the  character $\psi$. 
If $FJ_{\psi}(\Pi)$ is a Whittaker-generic representation of $M(\BA)$,
then there exists $f\in \Pi$ and $u^*=(n^*, 0,1)$ where $n^*\in k^{\times}$ such that 
\[ 
\int_{U(k)\backslash U(\mathbb{A})} f(u) \bar{\psi}_{u*}(u) du \neq 0. 
\] 
\end{prop} 

\begin{proof} For every $n^*\in k$ let $\psi_{n^*}$ be a character  of  $N(\BA)$ defined by $\psi_{n^*}(n)=\psi(n^*x)$ where
$n=\left(\begin{smallmatrix} 1 & x \\ 0 & 1 \end{smallmatrix}\right)$. 
Since $FJ_{\psi}(\Pi)$ is a Whittaker-generic representation of $M(\BA)$, there is $f' \in \Pi$, $\phi \in S(Y(\BA))$ and $n^* \in k^{\times}$, such that the following integral is non-vanishing:
\[
\int_{N(k) \bs N(\BA)} \int_{X(k) \bs V(\BA)} f'(vn) [\rho_{\bar\psi}(vn) \phi](1) \ol{\psi}_{n*}(n)dv dn \neq 0.
\] 
Since $X(k)\bs V(\BA)$ is  a union of $X(k)\bs X(\BA) \cdot y$, where $y$ runs over $Y(\BA)$, there exists $y\in Y(\BA)$ such that 
\[
\int_{N(k) \bs N(\BA)} \int_{X(k) \bs X(\BA)} f'(xyn)[{\rho}_{\bar\psi}(xyn)\phi](1)\ol{\psi}_{n*}(n)dx dn \neq 0. 
\]
Since $yn=x'ny$, for some element $x'\in X(\BA)$, after changing the variable $x$, the previous expression is equivalent to 
\[
\int_{N(k) \bs N(\BA)} \int_{X(k) \bs X(\BA)} f'(xny)[{\rho}_{\bar\psi}(xny)\phi](1)\ol{\psi}_{n*}(n)dx dn \neq 0.  
\]
Since $[{\rho}_{\bar\psi}(xny)\phi](1)= \phi(y)$, it follows that 
\[
\int_{N(k) \bs N(\BA)} \int_{X(k) \bs X(\BA)} f'(xny)\ol{\psi}_{n*}(n)dx dn \neq 0.  
\]
Let $f$ be the right translate of $f'$ by $y$. The conclusion of the proposition holds for this $f$. 

\end{proof}

\section{Nilpotent orbits} 
  Let $G$ be a central extension of a reductive group over  a local field $k$.
Let $\frak{g}$ be the Lie algebra of $G$. Let $u$ be a nilpotent element in $\frak{g}$. 
Let $\varphi : \frak{sl}_2 \rightarrow \frak{g}$  be a homomorphism corresponding to $u$ by the Jacobson-Morozov theorem. 
 Let $\psi$ be a smooth  character of $k$ and  $\kappa$ the Killing form on $\frak{g}$. Then $u$ defines a function 
 \[ 
 \psi_u: \frak{g}  \rightarrow \mathbb C^{\times} 
 \] 
  by 
 \begin{equation}\label{eq1} 
 \psi_u(x) =\psi (\kappa (u, x)).
  \end{equation}  
 
  Let $\frak{n}$ be a nilpotent subalgebra of $\frak{g}$. If $\kappa(u,[x,y])=0$ for all $x,y\in \frak{n}$ then $\exp(x) \mapsto \psi_u(x)$ defines a character of 
  $N=\exp(\frak{n})$. Abusing notation, we shall use $\psi_u$ to denote this character of $N$. 
  One prominent example arises as follows. Let 
\[  
s = \varphi \left(\begin{matrix}
1 & 0 \\
0 & -1
\end{matrix}\right).
\] 
Let $\frak{g}(j)=\{ x\in \frak{g} ~|~ [s,x]= j\cdot x\}$ be the $s$-{\em weight $j$ space}.  Let 
\[ 
\frak{n}_u=\oplus_{j\geq 2}  \frak{g}(j).
\] 
Since $u\in \frak{g}(-2)$ and two weight spaces are perpendicular unless the weights are opposite, it follows that $\kappa(u,[x,y])=0$ for all $x,y\in \frak{n}_u$. Hence $\psi_u$ defines a character of $N_u=\exp(\frak{n}_u)$. The pair $(N_u, \psi_u)$ is 
precisely the one discussed in the introduction. 

In a similar fashion,  if $k$ is a global field and $\psi$ a smooth character of $\mathbb A$ trivial on $k$, then any
nilpotent element  $u$ in $\frak{g}$ defines a character $\psi_u$ of $N_u(\mathbb A)$ trivial on $N_u(k)$.

\section{Raising nilpotent orbits}  \label{S:raising}

We continue with the set up of the previous section. 
 Let $\mathfrak c \subseteq \mathfrak g$ be the centralizer of 
$\varphi(\mathfrak{sl}_2)$ in $\mathfrak g$. Assume we have a non-trivial map 
\[ 
\varphi_{\mathfrak c} : \mathfrak{sl}_2 \rightarrow \mathfrak c. 
\] 
Let 
\[  
s_{\mathfrak c} =\varphi_{\mathfrak c}  \left(\begin{matrix}
1 & 0 \\
0 & -1
\end{matrix}\right) 
\text{ and } 
u_{\mathfrak c} =\varphi_{\mathfrak c}  \left(\begin{matrix}
0 & 0 \\
\nu & 0
\end{matrix}\right) 
\] 
for some $\nu \neq 0$. The map $\varphi_{\mathfrak c}$ lifts to a map $\varphi_{\mathfrak c} : \widetilde{\SL}_2 \rightarrow G$ where $\widetilde{\SL}_2$ is a central extension 
of $\SL_2$. 

  Let $\mathfrak g(j,l)$ be the subspace of $\mathfrak{g}$  of all elements of 
  $s$-weight $j$ and $s_{\mathfrak c}$-weight $l$. 
  Let $u'=u+u_{\frak{c}}$. Note that $\frak{n}_u$ is a sum of $\frak{g}(j,l)$ such that $j\geq 2$, while 
$\frak{n}_{u'}$ is a sum of $\frak{g}(j,l)$ such that $j+l\geq 2$.

Assume that 
\begin{enumerate}
\item The $s_{\mathfrak c}$-weights $l$ are bounded by 2. 
\item Under the action of $\varphi_{\mathfrak c}(\frak{sl}_{2})$ 
\[ 
\mathfrak g(1)= \mathfrak g(1)^{\varphi_{\mathfrak c}(\frak{sl}_{2})} \oplus m V_2. 
\] 
\item $\dim \mathfrak g(0,2) = \dim \frak{g}(2,2) + 1$. 
\end{enumerate}  
Here $V_2$ denotes the irreducible $2$-dimensional representation of $\frak{sl}_{2}$. 
 Let $\log X=\frak{g}(1,1)$ and $\log Y=\frak{g}(1,-1)$. Let 
$\frak{h}= \log X\oplus  \log Y \oplus \frak{n}_u$ 
and $H=\exp(\frak{h})$.
  Let $N'_u$ be the co-dimension one subgroup of $N_u$ such that $\psi_u$ is trivial on $N'_u$. 
Then $H/ N'_u$ is a Heisenberg group with the center $N_u/N'_u$. Let $J$ be the semi direct product  of $H/ N'_u$  and $M$ where $M$ is (a central extension of) 
 $\varphi_{\mathfrak c}(\widetilde{\SL}_{2})$. 
  It is a Jacobi group. Let $\pi$ be a smooth representation of $G$. Then $\pi_{N_u, \psi_u}$ is a representation of $J$.
Assume that $\pi_{N_u, \psi_u}\neq 0$. 
If the natural action of $M$ on  $FJ_{\psi_u}(\pi_{N_u, \psi_u})=\Hom_{H}(\rho_{\psi_u}, \pi_{N_u, \psi_u})$ is Whittaker-generic, i.e. not a multiple of the trivial 
representation,  then we can raise the orbit. More precisely, we have the following:

\begin{prop}\label{P5:local} 
 Assume that $k$ is a non-archimedean local field. Assume that the conditions (1-3) above are all satisfied. 
Let $\pi$ be a smooth representation of $G$ such that $\pi_{N_u, \psi_u}\neq 0$. Assume that $FJ_{\psi_u}(\pi_{N_u, \psi_u})$ 
 is a Whittaker-generic representation of $M$. Then $\pi_{N_{u'}, \psi_{u'}}\neq 0$, where $u'=u+u_{\mathfrak c}$, 
for some choice of  nilpotent $u_{\mathfrak c}\neq 0$ in $\varphi_{\mathfrak c}(\frak{sl}_{2})$.
\end{prop}

\begin{proof}
 Let  $\mathfrak{n} \subset \varphi_{\mathfrak c}( \mathfrak{sl}_{2})$ be the subspace of 
  elements of $s_{\mathfrak c}$-weight 2 and let $N= \exp (\mathfrak{n}) \subset M$. 
  Let $\frak{u}= \frak{n} \oplus \log X \oplus \mathfrak{n}_{u}$ and  $U=\exp(\frak{u})$.

Since $FJ_{\psi_u}(\pi_{N_u, \psi_u})$ is a Whittaker-generic representation of $M$, by Proposition \ref{P2:local},
$\pi_{U, \tilde{\psi}}\neq 0$ for some character $\tilde{\psi}$ of $U$, equal to $\psi_u$ on $N_u$, trivial on $X$ and equal to 
$\psi_{u_{\mathfrak c}}$ on $N$ for some choice of  
nilpotent $u_{\mathfrak c}\neq 0$ in $\varphi_{\mathfrak c}(\frak{sl}_{2})$.  In the remainder of the proof we shall ``transfer'' the character $\tilde{\psi}$ from $U$ to $N_{u'}$ using a Heisenberg group that appears as a quotient of $UN_{u'}$. 

Consider the sum $\oplus_j \frak{g}(j,2)$ over all $j$.  It is an $\varphi(\frak{sl}_2)$-module. By representation theory of 
$\frak{sl}_2$, the map $x \rightarrow [u,x]$ is an injection of $\frak{g}(2,2)$ into $\frak{g}(0,2)$ and the complement of the image is spanned by 
$\varphi(\frak{sl}_2)$-fixed vectors. Since $\mathfrak{n}_{\mathfrak c}$ is fixed by $\varphi(\frak{sl}_2)$ and 
 $\dim \mathfrak g(0,2) = \dim \frak{g}(2,2) + 1$ it follows that 
\[ 
\mathfrak g(0,2) = \mathfrak{n}_{\mathfrak c} \oplus [u, \frak{g}(2,2)].
\]  
   Let $\frak{u}'$ be the space obtained from $\frak{u}$ by removing 
 the summand $\frak{g}(2,-2)$ and adding $[u, \frak{g}(2,2)]$, so  $\frak{u}'$ is a direct sum of 
 $\frak{g}(0,2)$, $\frak{g}(1,1)$ and $\frak{g}(j,l)$ for all $j\geq 2$ but not $(j,l)=(2,-2)$. 
Since $\frak{n}_{u'}$ is a sum of  $\frak{g}(j,l)$ with $j+l \geq 2$,  and $|l| \leq 2$ by the first assumption,  it follows that $\frak{u}'$ is a subalgebra containing 
 $\frak{n}_{u'}$. 
 
 \begin{lem} \label{L5:char} Let $U'=\exp(\frak{u}')$. Then $\psi_{u'}$ is a character of $U'$, equal to $\tilde{\psi}$ on $U'\cap U$. 
  \end{lem} 
 \begin{proof} To prove that $\psi_{u'}$ is a character, we need to check $\kappa(u',[x,y])=0$ for all $x,y\in \frak{u}'$. Clearly, It suffices to prove vanishing when $x\in\frak{g}(j,l)$ and $y\in\frak{g}(j',l')$. 
Recall that  $u\in \frak{g}(-2,0)$ and $u_{\mathfrak c}\in \frak{g}(0,-2)$. If $\kappa(u_{\mathfrak c}, [x,y])\neq 0$, then $j+j'=0$. Hence $j=j'=0$ and $x,y\in \frak{g}(0,2)$. But then $[x,y]=0$, a 
contradiction.  If $\kappa(u, [x,y])\neq 0$, then $j+j'=2$ and $l+l'=0$. Hence up to permutation, $x\in\frak{g}(0,2)$ and $y\in\frak{g}(2,-2)$. 
Again a contradiction, since $\frak{g}(2,-2)$ is not in $\frak{u}'$.  Hence $\psi_{u'}$ is a character of $U'$. 

Note that $U'\cap U= N_{\mathfrak c}X ( N_u \cap U')$. Since both, $\psi_{u'}$ and $\tilde{\psi}$,  are equal to 
$\psi_{u_{\mathfrak c}}$, $1$ and $\psi_u$ on the three factors, respectively, it follows that  $\psi_{u'}= \tilde{\psi}$ on $U'\cap U$, as desired. 
\end{proof}

 Let  $\frak{z}'=\frak{u} \cap \frak{u}'$ and let 
 \[ 
 \frak h' = \frak{u} + \frak{u}' = [u, \frak{g}(2,2)] \oplus   \frak{g}(2,-2) \oplus \frak{z}'. 
 \]

\begin{lem} \label{L:non-deg} 
The pairing $\kappa(u',[x,y])$ where  $x\in [u, \frak{g}(2,2)]$ and $y\in \frak{g}(2,-2)$ is non-degenerate. 
\end{lem} 
\begin{proof}  Since $[x,y]\in \frak{g}(2)$ and $u_i \in \frak{g}(0)$, 
$\kappa(u',[x,y]) =\kappa(u,[x,y])$.  By the invariance of the Killing form, we have 
\[
\kappa(u,[x,y])  =\kappa(y,[u,x]). 
\] 
The pairing is non-degenerate since $[u,[u, \frak{g}(2,2)]]=\frak{g}(-2,2)$,  and the Killing form gives a non-degenerate paring between opposite weight spaces. 
\end{proof} 

Let $H'=\exp(\frak{h}')$ and $Z'=\exp(\frak{z'})$, i.e. $Z'=U\cap U'$.  
Let $Z''$ be the co-dimension one subgroup of $Z'$ where the character $\psi_{u'}=\tilde{\psi}$ is trivial.  
Lemma  \ref{L:non-deg} implies that $H'/Z''$ is a Heisenberg group. Since $\pi_{U,\tilde{\psi}}\neq 0$, it follows that 
 \[ 
\pi_{Z',\tilde{\psi}}=  \pi_{Z',\psi_{u'}}\neq 0. 
\] 
 Proposition \ref{P1:local} implies that 
$\pi_{U', \tilde{\psi}'}\neq 0$ for any character $\tilde{\psi}'$ of $U'$ extending $\psi_{u'}$ on $Z'$, in particular $\psi_{u'}$.  
  Since $N_{u'} \subseteq U'$, the proposition follows. 
\end{proof} 

Recall that $G$ is a central extension of a linear group. The map  $\varphi_{\mathfrak c}$ lifts to a map $\varphi_{\mathfrak c} : \widetilde{\SL}_2 \rightarrow G$ where 
$\widetilde{\SL}_2$ is a central extension 
of $\SL_2$ by $\mu_d$, the group of $d$-th roots of 1. A representation of $\widetilde{\SL}_2$ is called {\em genuine} if $\mu_d$ acts faithfully. 

\begin{cor}\label{cor:local}
Assume that $k$ is a non-archimedean local field. 
 Assume that the conditions (1-3) above are all satisfied. 
Let $\pi$ be a smooth representation of $G$ such that $\pi_{N_u, \psi_u}\neq 0$.  
Assume, furthermore,  that the restriction  of $\pi$  to $\varphi_{\mathfrak c}(\widetilde{\SL}_2)$ is genuine. 
If  $d$ is $1$ and $m$ is odd, or $d$ is $2$ and $m$ is even, or $d > 2$ and $m$ is arbitrary, then 
 $\pi_{N_{u'}, \psi_{u'}}\neq 0$, where $u'=u+u_{\mathfrak c}$, for some choice of  
nilpotent $u_{\mathfrak c}\neq 0$ in $\varphi_{\mathfrak c}(\frak{sl}_{2})$.
\end{cor}

\begin{proof} By properties of representations of Jacobi groups, the group $M$ acts on $\rho_{\psi_u}$ by its linear quotient if $m$ is even, or 2-fold central extension if $m$ is odd. Thus the conditions assure that 
$FJ_{\psi_u}(\pi_{N_u, \psi_u})$ is a genuine representation of non-trivial central extension of $\SL_2$. 
Hence, it is a Whittaker-generic representation of $M$ and Corollary follows from Proposition \ref{P5:local}.
\end{proof}

Now assume that $k$ is a global field. Let $\Pi$ be a space of smooth functions on $G(k) \backslash G(\mathbb A)$ stable under the action of 
$G(\mathbb A)$ by right translations. Let $\Pi_{N_u, \psi_u}$ be the space of smooth functions on $J(k) \backslash J(\mathbb A)$ consisting of 
\[ 
\tilde{f}(g)= \int_{N_u(k)\backslash N_u(\mathbb A)} f(ng) \bar{\psi}_u(n) dn
\]
where $f$ runs through $\Pi$. Similarly one can define $\Pi_{U, \tilde{\psi}}$, where $\tilde{\psi}$ is a character of $U$.
Recall from Section 5 that one then can define $FJ_{\psi_u}(\Pi_{N_u, \psi_u})$ which is the representation of $M(\BA)$.

The following global version of Proposition \ref{P5:local} is proved 
 using Propositions \ref{P1:global} and \ref{P2:global}  (instead of Propositions \ref{P1:local} and \ref{P2:local}). We omit the details.

\begin{prop} \label{P5:global} 
Assume that $k$ is a global field.
Assume that the conditions (1-3) above are all satisfied. 
Let $\Pi$ be a a space of smooth functions on $G(k) \backslash G(\mathbb A)$ stable under the action of 
$G(\mathbb A)$ by right translations such that  
$\Pi_{N_u, \psi_u}\neq 0$. 
Assume that $FJ_{\psi_u}(\Pi_{N_u, \psi_u})$ is a Whittaker-generic representation of $M(\BA)$. 
Then $\Pi_{N_{u'}, \psi_{u'}} \neq 0$, 
where $u'=u+u_{\frak{c}}$, for some choice of  
nilpotent $u_{\mathfrak c}\neq 0$ in $\frak{sl}_{2,{\mathfrak c}}$. 
\end{prop} 

We also have the following global analogue of Corollary \ref{cor:local}, with a similar argument (using Proposition \ref{P5:global} instead of Proposition \ref{P5:local}).

\begin{cor}\label{cor:global}
Assume that $k$ is a global field and $G(\BA)$ is a $d$-fold central extension of a connected reductive group. 
Assume that the conditions (1-3) above are all satisfied. 
Let $\Pi$ be a a space of smooth functions on $G(k) \backslash G(\mathbb A)$ stable under the action of 
$G(\mathbb A)$ by right translations such that  
$\Pi_{N_u, \psi_u}\neq 0$. 
Assume further that either $d$ is $1$ and $m$ is odd, or $d$ is $2$ and $m$ is even, or $d > 2$ and $m$ is arbitrary.
Then $\Pi_{N_{u'}, \psi_{u'}} \neq 0$, 
where $u'=u+u_{\frak{c}}$, for some choice of  
nilpotent $u_{\mathfrak c}\neq 0$ in $\frak{sl}_{2,{\mathfrak c}}$.
\end{cor}

\section{Symplectic-orthogonal groups} 
Let $W$ be a vector space over a $p$-adic field $k$ equipped with a non-degenerate symmetric or skew-symmetric bilinear form $\langle , \rangle$. 
Let $\frak{g}$ be the Lie algebra of endomorphisms of $W$ preserving the bilinear form.  Let 
\[ 
i : W \otimes W \rightarrow \End(W)
\] 
be a map defined by 
\[ 
i(x\otimes y)(z) = \langle y,z\rangle x -  \langle z, x\rangle y 
\] 
for all $x,y,z\in W$.  
 If  $\langle , \rangle$ is symmetric then $i$ gives a bijection between  $\wedge^2 W$,  the exterior square of $W$,  and  $\frak{g}$.  
If $\langle , \rangle$ is skew-symmetric then then $i$ gives a bijection between $S^2 (W)$,  the symmetric  square of $W$,  and  $\frak{g}$.

 Let $V_j$ be irreducible representation of 
$\frak{sl}_2$ of dimension $j$. Then $V_j$ admits a unique, up to a non-zero scalar, $\frak{sl}_2$-invariant bilinear form $\langle , \rangle_j$. 
This form is symmetric if $j$ is odd, and skew-symmetric if $j$ is even. We normalize $\langle , \rangle_j$ as follows. 
Let $v_j\in V_j$ be a non-zero highest weight vector. Let $u=\left(\begin{smallmatrix} 0 & 0 \\ 1 & 0 \end{smallmatrix}\right)$. 
Then $u^{j-1} v_j$  is a lowest weight vector. Hence the number 
\[ 
a=\langle v_j, u^{j-1} v_j\rangle_j 
\] 
 is non-zero. This number depends on the choice of $v_j$, however, its square class does not. 
In order to distinguish the forms, we shall write $\langle,\rangle_j^a$ for this form or simply $\langle , \rangle_j$ if  $a$ is in the class of 1. This normalization differs from the one in \cite{W01} by the factor $(-1)^{[(i-1)/2]}$.

Let $\varphi : \frak{sl}_2 \rightarrow \mathfrak g$. 
Then, under the action of $\varphi(\frak{sl}_2)$, the symplectic space $W$ can be decomposed as 
\begin{align}\label{eq3}
\begin{split}
W & = \oplus_j V_j \otimes U_j,\\
\langle , \rangle & =\oplus_j \langle,\rangle_j \otimes ( , )_j
\end{split}
\end{align}
where 
\[ 
U_j = \Hom_{\frak{sl}_2}(V_j,V)
\] 
 with the form $(,)_j$. The nilpotent orbit of $u$ gives  a partition $\ul{p}$ of $\dim(W)$ in which $j$ appears with multiplicity $\dim U_j$. 
If the form on $W$ is symmetric then the forms $(,)_j$ and $\langle , \rangle_j$  have the same signs. Otherwise the two forms have  
 different  signs.  If $(,)_j$ is skew-symmetric then $\dim(U_j)$ is even. Thus $j$ in $\ul{p}$ must have even multiplicity. The conjugacy class of $u$ is determined 
 by the partition $\ul{p}$ and isomorphism classes of $(U_j, (,)_j)$.

Put $W_j=U_j\otimes V_j$, in the decomposition (\ref{eq3}). 
If $\langle , \rangle$ is symmetric, then 
\begin{align}\label{eq4}
\begin{split} 
\frak{g}\cong \wedge^2(W) \cong (\oplus_j\wedge^2(W_j)) \oplus (\oplus_{i<j} W_i \otimes W_j),\\
\wedge^2(W_j)\cong \wedge^2(V_j) \otimes S^2(U_j)\oplus S^2(V_j) \otimes \wedge^2(U_j). 
\end{split} 
\end{align}

If $\langle , \rangle$ is skew-symmetric, then 
\begin{align}\label{eq5}
\begin{split} 
\frak{g}\cong S^2(W) \cong (\oplus_i S^2(W_j)) \oplus (\oplus_{i<j} W_i \otimes W_j), \\
S^2(W_j)\cong S^2(V_j) \otimes S^2(U_j)\oplus \wedge^2(V_j) \otimes \wedge^2(U_j). 
\end{split}
\end{align} 

If the decomposition \eqref{eq3} of $W$ contains a non-trivial summand  such that the form $(,)_i$ is skew-symmetric, then 
$\Sp(U_i)$ acts naturally on $W$  preserving the form $\langle, \rangle$. This gives an inclusion of $\Sp(U_i)$ into $G$. The adjoint action of 
$\Sp(U_i)$ on $\mathfrak g$ commutes with $\varphi(\mathfrak{sl}_2)$. In particular,  $\Sp(U_i)$  acts on each $s$-weight space $\mathfrak g(j)$.

\begin{lem}\label{L:decomposition}
Assume that the decomposition \eqref{eq3} of $W$ contains a non-trivial summand  such that the form $(,)_i$ is skew-symmetric. 
Let $\varphi_{\frak{c}} : \SL_2 \rightarrow \Sp(U_i)$  be a map corresponding to a long root of $\Sp(U_i)$. 
Let  $\SL_{2,\frak{c}}$ be the image of $\varphi_{\frak{c}}$. 
Then, under the adjoint action of $\SL_{2,\frak{c}}$,  $\frak{g}(1)$ decomposes as 
$$\frak{g}(1)\cong \frak{g}(1)^{\SL_{2,\frak{c}}} \oplus m V_2,$$
where, if $i$ is odd, 
\[
m=i(\sum_{j>i, j \text{ even}} \dim U_j) \oplus (\sum_{j < i, j \text{ even}} j \dim U_j) 
\]
and, if $i$ is even, 
\[ 
m=i(\sum_{j>i, j \text{ odd}} \dim U_j) \oplus (\sum_{j < i, j \text{ odd}} j \dim U_j). 
\]

\end{lem}
\begin{proof} Under the action of $\SL_{2,\frak{c}}$, the space $U_i$ decomposes as $V_2\oplus V_2^c$ where $\SL_{2,\frak{c}}$ acts on $V_2^c$ trivially. 
 It follows from (\ref{eq4}) and  (\ref{eq5}) that 
\[ 
\oplus_{j\neq i}  ((V_i \otimes V_j) (1))\otimes (V_2 \otimes U_j)
\] 
is the complement of $\frak{g}(1)^{\SL_{2,\frak{c}}}$ in $\frak{g}(1)$.  
On the other hand,  $(V_i \otimes V_j) (1)\neq 0$ only  for $i$ and $j$ of different parity, and then 
\[ 
\dim ((V_i\otimes V_j )(1)) =\min\{i,j\}. 
\] 
The lemma follows. 
\end{proof}

If the decomposition \eqref{eq3} of $W$ contains a non-trivial summand  such that the form $(,)_i$ is symmetric, then 
$\mathrm{O}(U_i)$ acts naturally on $W$  preserving the form $\langle, \rangle$. This gives an inclusion of $\mathrm{O}(U_i)$ into $G$. The adjoint action of 
$\mathrm{O}(U_i)$ on $\mathfrak g$ commutes with $\varphi(\mathfrak{sl}_2)$. In particular,  $\mathrm{O}(U_i)$  acts on each $s$-weight space $\mathfrak g(j)$. 

\begin{lem}\label{L:decomposition-2}
Assume that the decomposition \eqref{eq3} of $W$ contains a non-trivial summand for $i$ such that the form $(,)_i$ is symmetric with dimension $\geq 4$, and has a two-dimensional isotropic sub-space. 
Then $\mathrm{O}(U_i)$ has a parabolic subgroup fixing this sub-space with Levi subgroup isomorphic to $\GL_2 \times \mathrm{O}(U_i')$.
Let $\varphi_\frak{c} : \SL_2 \rightarrow \mathrm{O}(U_i)$  be a map corresponding to the roots of $\GL_2$. 
Let  $\SL_{2,\frak{c}}$ be the image of $\varphi_{\frak{c}}$. 
Then the structure of the $\SL_{2,\frak{c}}$-module $\frak{g}(1)$ is given by 
$$\frak{g}(1)\cong \frak{g}(1)^{\SL_{2,\frak{c}}} \oplus 2m V_2,$$
where, if $i$ is odd, 
\[
m=i(\sum_{j>i, j \text{ even}} \dim U_j) \oplus (\sum_{j < i, j \text{ even}} j \dim U_j) 
\]
and, if $i$ is even, 
\[ 
m=i(\sum_{j>i, j \text{ odd}} \dim U_j) \oplus (\sum_{j < i, j \text{ odd}} j \dim U_j). 
\]

\end{lem}
\begin{proof} Under the action of $\SL_{2,\frak{c}}$, the space $U_i$ decomposes as $(V_2 \otimes U_2) \oplus (V_2 \otimes U_2)^c$ where $\SL_{2,\frak{c}}$ acts on $(V_2 \otimes U_2)^c$ trivially, and $\dim(U_2)=2$. 
 It follows from (\ref{eq4}) and  (\ref{eq5}) that 
\[ 
\oplus_{j\neq i}  ((V_i \otimes V_j) (1))\otimes ((V_2 \otimes U_2) \otimes U_j)
\] 
is the complement of $\frak{g}(1)^{\SL_{2,\frak{c}}}$ in $\frak{g}(1)$.  
On the other hand,  $(V_i \otimes V_j) (1)\neq 0$ only  for $i$ and $j$ of different parity, and then 
\[ 
\dim ((V_i\otimes V_j )(1)) =\min\{i,j\}. 
\] 
The lemma follows. 
\end{proof}

\section{Raising nilpotent orbits from symplectic stabilizers} 

Let $G=\Sp(W)$ or $\wt{\Sp}(W)$, if $W$ is a symplectic space,  or $G=\mathrm{O}(W)$ if $W$ is an orthogonal space.
Assume that we are in the situation as in Lemma \ref{L:decomposition}, that is, the decomposition \eqref{eq3} of $W$ contains a non-trivial summand for $i$ such that the form $(,)_i$ is skew-symmetric. 
Let $\varphi_{\frak{c}} : \frak{sl}_2 \rightarrow \frak{sp}(U_i)$  be a map corresponding to a long root of $\frak{sp}(U_i)$.  Let 
 \[ 
\varphi'=\varphi\oplus \varphi_{\frak{c}}  : \frak{sl}_2 \rightarrow \frak{g}
\] 
and $u'=u+u_{\frak{c}}$, a nilpotent element corresponding to $\varphi'$.  Since the orbit of $u_{\frak{c}}$ in $\frak{sl}_{2,{\frak{c}}} = \varphi_{\frak{c}} (\frak{sl}_{2})$ contains 0 in its closure, it follows that the closure of the orbit of $u'$ contains the orbit of $u$.  

We shall now determine the partition $\ul{p}'$ corresponding to $u'$. 
From the decomposition (\ref{eq3}) of $W$ under the action of $\varphi(\frak{sl}_2)$ one can easily obtain a similar decomposition for 
$\varphi'(\frak{sl}_2)$. Indeed, under the action of $\varphi_{\frak{c}}(\frak{sl}_2)$, the symplectic space $U_i$ decomposes $U_i=V_2 \oplus V_2^c$ where $V_2$ is the irreducible 2-dimensional representation of $\frak{sl}_2$, and $\varphi_{\frak{c}}(\frak{sl}_2)$ acts on the complement $V_2^c$ trivially. 
 Now the Clebsch-Gordan formula $V_i\otimes V_2\cong V_{i-1} \oplus V_{i+1}$ implies that the partition $\ul{p'}$ is obtained from the partition $\ul{p}$, corresponding to $u$,  by replacing a pair $(i,i)$ by $(i+1,i-1)$. 
  
   A more refined information about the conjugacy class of $u'$ is given by the following proposition.

 \begin{prop} \label{P:forms}  Let $U_j$ be the spaces in the decomposition of $W$ under the action of $\varphi(\frak{sl}_2)$.  Then  
 the spaces $U'_j$ in the decomposition of $W$ under the action of $\varphi'(\frak{sl}_2)$ are 
  \[ 
  U'_i= V_2^c, \,  U'_{i+1}= U_{i+1}\oplus \langle b\rangle,  \, U'_{i-1}= U_{i-1}\oplus \langle b\rangle ,
  \] 
 for some $b\in k^{\times}$,  and $U'_j=U_j$ otherwise. Here $\langle b\rangle$ is the one-dimensional orthogonal space with the form $bx^2$. 
  \end{prop} 
  \begin{proof} 
   Let $v_2\in V_2 \subseteq U_i$ be a highest weight vector. Let $a=(v_2, u_{\frak{c}} v_2)_i$, so that the restriction of $(,)_i$ to $V_2$ is the form previously denoted as 
 $\langle, \rangle_2^a$. Proposition follows from the following lemma, with $b=ai$. 
 \end{proof} 
 
 \begin{lem} \label{L:forms} 
 With respect to the Clebsh-Gordan decomposition $V_i\otimes V_2\cong V_{i-1} \oplus V_{i+1}$, the form $\langle, \rangle_i \otimes \langle, \rangle^a_2$ 
 decomposes as 
 \[ 
 \langle, \rangle_i \otimes \langle, \rangle^a_2= \langle, \rangle_{i+1} ^{ai} \oplus  \langle, \rangle_{i-1} ^{ai}  
 \] 
 \end{lem} 
 \begin{proof} Let $v_i\in V_i$ be a highest weight vector such that $\langle v_i, u^{i-1} v_i\rangle =1$. Then 
 \[ 
 v_{i+1}= v_i \otimes v_2
 \]  is a highest weight vector of $V_{i+1}$. Using the Newton-Leibniz rule,  
 \[ 
 (u+u_{\frak{c}})^i v_{i+1} = i \cdot (u^{i-1} v_i \otimes u_{\frak{c}} v_2). 
 \] 
 It follows that the pairing  $\langle, \rangle_i \otimes \langle, \rangle^a_2$ evaluated at $v_{i+1}$ and $(u+u_{\frak{c}})^i v_{i+1}$ gives $ai$.  Similarly, 
 \[ 
 v_{i-1}= (i-1) \cdot (v_{i}\otimes u_{\frak{c}} v_2)  - u v_i \otimes v_2
 \]  is a highest weight vector of $V_{i-1}$. Then 
 \[ 
 (u+u_{\frak{c}})^{i-2} v_{i-1} = u^{i-2} v_i \otimes u_{\frak{c}} v_2 - u^{i-1} v_i \otimes v_2. 
 \] 
 Next, using $\langle -u v_i , u^{i-2} v_i\rangle_i = \langle v_i, u^{i-1} v_i\rangle_i =1$, the pairing  $\langle, \rangle_i \otimes \langle, \rangle^a_2$ evaluated at 
 $v_{i-1}$ and $(u+u_{\frak{c}})^{i-2} v_{i-1}$ also gives $ai$. 
  \end{proof}

We shall now describe $\frak{g}(j,l)$. 
Using the decompositions (\ref{eq4}),  (\ref{eq5}) and $U_i=V_2\oplus V_2^c$ we see that $\frak{g}(j,l)\neq 0$ only for $|l|\leq 2$. In particular, 
$\frak{g}(1,1)$ and $\frak{g}(0,2)$ are the only two such spaces contained in $\frak{n}_{u'}$ but not in  $\frak{n}_{u}$. 
Let
\[ 
E_i= \begin{cases}  \wedge^2 V_i \text{ if $i$ is even} \\ 
S^2 V_i \text{ if $i$ is odd.}
\end{cases} 
\] 
 Then
 \begin{equation}\label{eq10} 
 \frak{g}(j,\pm 2) \cong E_i(j) \otimes S^2(V_2)(\pm 2). 
 \end{equation} 
 
 From \eqref{eq10}, one can easily see that $\dim \frak{g}(0,2) = \dim \frak{g}(2,2) +1$, that is,  the condition (3) in Section \ref{S:raising}  is satisfied. Therefore, by Lemma \ref{L:decomposition}, all conditions (1-3) in Section \ref{S:raising} are satisfied. 
 The following proposition follows  from Proposition \ref{P5:local}.

 \begin{prop}\label{proplocal:sym}
 Assume that we are in the situation as in Lemma \ref{L:decomposition} and $k$ is a non-archimedean local field. Let $\pi$ be a smooth representation of $G$. If $G=\wt{\Sp}(W)$ assume that $\pi$ is a
genuine representation. Let $u$ be a nilpotent element in $\frak{g}$ such that $\pi_{N_u,\psi_u}\neq 0$.
Assume that $FJ_{\psi_u}(\pi_{N_u, \psi_u})=\Hom_{H_{u, \frak{c}}}(\rho_{\psi_u}, \pi_{N_u, \psi_u})$ is a Whittaker-generic representation of $M$.
Then $\pi_{N_{u'},\psi_{u'}}\neq 0$ for a nilpotent element $u'$ whose partition is obtained from the partition corresponding to $u$ via replacing a pair $(i,i)$ by $(i+1, i-1)$.
\end{prop}


There is also a global version of Proposition \ref{proplocal:sym}, which follows from Proposition \ref{P5:global}.  

\begin{prop}\label{propglobal:sym}
 Assume that we are in the situation as in Lemma \ref{L:decomposition} and $k$ is a global field. Let $\Pi$ be a space of smooth functions on 
 $G(k)\backslash G(\mathbb A)$ stable under the action of $G(\mathbb A)$ by right translations.  If $G=\wt{\Sp}(W)$ assume that $\Pi$ consists of 
 genuine functions.  Let $u$ be a nilpotent element in $\frak{g}$ such that $\Pi_{N_u,\psi_u}\neq 0$.
Assume that $FJ_{\psi_u}(\Pi_{N_u, \psi_u})$ (see Section 5 for definition) is a Whittaker-generic representation of $M(\BA)$.
Then 
$\Pi_{N_{u'},\psi_{u'}}\neq 0$ 
for a nilpotent element $u'$ whose partition is obtained from the partition corresponding to $u$ via replacing a pair $(i,i)$ by $(i+1, i-1)$.
\end{prop}

\section{Raising nilpotent orbits from orthogonal stabilizers} 
Let $G=\Sp(W)$ or $\wt{\Sp}(W)$, if $W$ is a symplectic space,  or $G=\mathrm{O}(W)$ if $W$ is an orthogonal space.
Assume that we are in the situation as in Lemma \ref{L:decomposition-2}, that is, the decomposition \eqref{eq3} of $W$ contains a non-trivial summand for $i$ such that the form $(,)_i$ is symmetric with dimension $\geq 4$, and has a two-dimensional isotropic sub-space. 
Then $O(U_i)$ has a parabolic subgroup fixing this sub-space with Levi subgroup isomorphic to $\GL_2 \times O(U_i')$.
Let $\varphi_\frak{c} : \SL_2 \rightarrow O(U_i)$  be a map corresponding to the roots of $\GL_2$. 
Let 
 \[ 
\varphi'=\varphi\oplus \varphi_{\frak{c}}  : \frak{sl}_2 \rightarrow \frak{g}
\] 
and $u'=u+u_{\frak{c}}$, a nilpotent element corresponding to $\varphi'$.  Since the orbit of $u_{\frak{c}}$ in $\frak{sl}_{2,{\frak{c}}} = \varphi_{\frak{c}} (\frak{sl}_{2})$ contains 0 in its closure, it follows that the closure of the orbit of $u'$ contains the orbit of $u$.

By the Clebsch-Gordan formula 
$$V_i\otimes (V_2 \otimes U_2)\cong V_{i-1}\otimes U_2 \oplus V_{i+1}\otimes U_2,$$
which implies that the partition $\ul{p'}$ corresponding to $u'$ is obtained from the partition $\ul{p}$, corresponding to $u$,  by replacing a quadruple $(i,i,i,i)$ by $(i+1,i+1,i-1,i-1)$.

Next, let 
\[ 
E_i= \begin{cases}  \wedge^2 V_i \text{ if $i$ is even} \\ 
S^2 V_i \text{ if $i$ is odd,}
\end{cases} 
\] 
\[ 
F_i= \begin{cases}  S^2 V_i \text{ if $i$ is even} \\ 
\wedge^2 V_i \text{ if $i$ is odd.}
\end{cases} 
\] 
 Then
 \begin{equation}\label{eq10-2} 
 \frak{g}(j,\pm 2) \cong E_i(j) \otimes \wedge^2(V_2 \otimes U_2)(\pm 2) \oplus F_i(j) \otimes S^2(V_2 \otimes U_2)(\pm 2). 
 \end{equation} 
 
From \eqref{eq10-2}, one can easily see that
 $\dim \frak{g}(0,2) = \dim \frak{g}(2,2) +1$, that is, the condition (3) in Section \ref{S:raising}  is satisfied. Therefore, by Lemma \ref{L:decomposition-2}, all conditions (1-3) in Section \ref{S:raising}  are satisfied.
 The following proposition follows directly from Proposition \ref{P5:local} and the discussion at the beginning of this section.

\begin{prop}\label{proplocal:ortho}
 Assume that we are in the situation as in Lemma \ref{L:decomposition-2} and $k$ is a non-archimedean local field. 
Let $\pi$ be a smooth representation of $G$.
If $G=\wt{\Sp}(W)$ assume that $\pi$ is a
genuine representation.
Let $u$ be a nilpotent element in $\frak{g}$ such that $\pi_{N_u,\psi_u}\neq 0$.
Assume that $FJ_{\psi_u}(\pi_{N_u, \psi_u})=\Hom_{H_{u, \frak{c}}}(\rho_{\psi_u}, \pi_{N_u, \psi_u})$ is a Whittaker-generic representation of $M$.
Then $\pi_{N_{u'}, \psi_{u'}}\neq 0$
for a nilpotent element $u'$ whose partition is obtained from the partition corresponding to $u$ via replacing a quadruple $(i,i,i,i)$ by $(i+1,i+1,i-1,i-1)$.
\end{prop}

There is also a global version of Proposition \ref{proplocal:ortho}, which follows directly from Proposition \ref{P5:global} and the discussion at the beginning of this section.

\begin{prop}\label{propglobal:ortho}
 Assume that we are in the situation as in Lemma \ref{L:decomposition-2} and $k$ is a global field. Let $\Pi$ be a space of smooth functions on 
 $G(k)\backslash G(\mathbb A)$ stable under the action of $G(\mathbb A)$ by right translations.  If $G=\wt{\Sp}(W)$ assume that $\Pi$ consists of 
 genuine functions.  
Let $u$ be a nilpotent element in $\frak{g}$ such that $\Pi_{N_u,\psi_u}\neq 0$.
Assume that $FJ_{\psi_u}(\Pi_{N_u, \psi_u})$ (see Section 5 for definition) is a Whittaker-generic representation of $M(\BA)$.
Then
$\Pi_{N_{u'},\psi_{u'}}\neq 0$  for a nilpotent element $u'$ whose partition is obtained from the partition corresponding to $u$ via replacing a quadruple $(i,i,i,i)$ by $(i+1,i+1,i-1,i-1)$.
\end{prop}

\section{Special orbits of classical groups} 

In this section we recall well known definitions of special orbits and introduce a notion of metaplectic-special orbits for $G=\wt{\Sp}(W)$.

\begin{defn}

A symplectic orbit is  called \textbf{symplectic special} if the number of even parts bigger than every odd $i$ appearing in the corresponding partition is even. 

A symplectic orbit is called \textbf{metaplectic special} if the number of even parts bigger than every odd $i$ appearing in the corresponding partition is odd. 

An orthogonal  orbit  is called \textbf{orthogonal special} if the number of odd parts smaller than every even $i$ appearing in the corresponding partition is even.

\end{defn}

\begin{prop}\label{P:special} 

A symplectic orbit is  \textbf{symplectic special} if the number $m$ given by   Lemma \ref{L:decomposition}   is even for every odd $i$ appearing in the corresponding partition. 

A symplectic orbit is  \textbf{metaplectic special} if the number  $m$ given by  Lemma \ref{L:decomposition} is odd for  every odd $i$ appearing in the corresponding partition. 

An orthogonal orbit  is  \textbf{orthogonal special} if the number $m$ given by Lemma \ref{L:decomposition} is even for every even $i$ appearing in the corresponding partition. 

\end{prop} 
\begin{proof}  In the first two cases, when  $i$ is odd, then the parity of $m$ depends on the parity of 
\[ 
\sum_{j>i, j \text{ even }} \dim U_j 
\] 
but this number is exactly the number of even parts greater than $i$. In the third case, when $i$ is even, than the parity of $m$ depends on the parity of 
\[ 
\sum_{j<i, j \text{ odd }} \dim U_j 
\] 
but this number is exactly the number of odd parts less than $i$

\end{proof}

Given a symplectic partition $\ul{p}$ of $2n$, one easily checks that $\ul{p}$ is metaplectic special if and only if the transpose of $\ul{p}$
 is an orthogonal partition of $2n$.  Conversely, it is known that an orthogonal partition of $2n$ is special if and only if its transpose is a symplectic partition. Summarizing we have the following: 

\begin{prop} The transpose of partitions defines a bijection between metaplectic-special partitions of $2n$ and special orthogonal partitions of $2n$. 
\end{prop} 

\begin{defn}

 Given any symplectic partition $\ul{p}$ of $2n$, its \textbf{special expansion} $\ul{p}^{\Sp}$ is defined to be the smallest symplectic special partition which is bigger than $\ul{p}$.

 Given any symplectic partition $\ul{p}$ of $2n$, its \textbf{metaplectic special expansion} $\ul{p}^{\wt{\Sp}}$ is defined to be the smallest metaplectic special  partition of  $2n$ which is bigger than $\ul{p}$.

 Given any orthogonal partition $\ul{p}$ of $m$, its \textbf{special expansion} $\ul{p}^{\mathrm{O}}$ is defined to be the smallest orthogonal special symplectic partition of  $m$ which is bigger than $\ul{p}$.

\end{defn}

Lemma 6.3.9 in \cite{CM93} gives a recipe for calculating special expansions. 
In the following proposition, we give a recipe for calculating the metaplectic-expansion. The proof of this proposition is very similar to that of  Lemma 6.3.9  in \cite{CM93} and will be omitted. 

\begin{prop}[Recipe for calculating metaplectic expansion]\label{recipe}
Given a symplectic partition $\ul{p}$ of $2n$, we may write $\ul{p}=[p_1 p_2 \cdots p_r]$
with $p_1 \geq p_2 \geq \cdots \geq p_r >0$. Enumerate the
indices $i$ such that $p_{2i-1} = p_{2i}$ is odd and
$p_{2i-2} \neq p_{2i-1}$ or $2i-2=0$ as $i_1 < \cdots < i_t$ (the set of indices $\{i_1, \ldots, i_t\}$ might be empty). Then
the metaplectic expansion of $\ul{p}$ can be obtained by
replacing each pair of parts $(p_{2i_j-1}, p_{2i_j})$
by $(p_{2i_j-1}+1, p_{2i_j}-1)$, respectively,
and leaving the other parts alone.
\end{prop}

\section{Wave-front sets of classical groups} 

 Let $G=\Sp(W)$ or $\wt{\Sp}(W)$, if $W$ is a symplectic space,  or $G=\mathrm{O}(W)$ if 
$W$ is an orthogonal space.  Let $u$ be a nilpotent element in $\frak{g}$, and $\ul{p}$ the corresponding partition of $\dim(W)$. 
We shall say that $\ul{p}$ is special if it is metaplectic, symplectic or orthogonal special, respectively.

\begin{thm}\label{thm1} Let $k$ be a non-archimedean local field. Let $\pi$ be a smooth representation of $G$. If $G=\wt{\Sp}(W)$ assume that $\pi$ is a
genuine representation. Let $u$ be a nilpotent element in $\frak{g}$ such that $\pi_{N_u,\psi_u}\neq 0$. Let $\ul{p}$ be the partition corresponding to $u$. 
Then $\pi_{N_{u'},\psi_{u'}}\neq 0$ for a nilpotent element $u'$ whose partition is the special expansion $\ul{p}^G$  of  $\ul{p}$. 
 \end{thm}

\begin{proof} Assume first that $\pi$ is a genuine representation of $\wt{\Sp}(W)$.  If the partition  $\ul{p}$ is not metaplectic-special then,
 by Proposition \ref{P:special}, there is an odd integer $i$ such that the number $m$ given in Lemma \ref{L:decomposition} is even.
 So, we are in the situation as in Corollary \ref{cor:local} and $FJ_{\psi_u}(\pi_{N_u, \psi_u})=\Hom_{H_{u, \frak{c}}}(\rho_{\psi_u}, \pi_{N_u, \psi_u})$ is a Whittaker-generic representation of $\SL_{2,\mathfrak{c}}$.
 Hence by Proposition \ref{proplocal:sym}, $\pi_{N_{u'},\psi_{u'}}\neq 0$ for a nilpotent element $u'$ whose partition is obtained from the partition of $u$ by replacing a pair $(i,i)$ by $(i+1,i-1)$. 
  We can repeat this procedure until the partition is metaplectic special. The cases when $\pi$ is a representation of 
 ${\Sp}(W)$ or  $\mathrm{O}(W)$ are proved analogously. The theorem is proved. 
   
\end{proof} 

Using Proposition \ref{propglobal:sym} (instead of Proposition \ref{proplocal:sym}) we can prove a global version of Theorem \ref{thm1}. 

 \begin{thm}\label{thm2} Let $k$ be a global field. Let $\Pi$ be a space of smooth functions on 
 $G(k)\backslash G(\mathbb A)$ stable under the action of $G(\mathbb A)$ by right translations.  If $G=\wt{\Sp}(W)$ assume that $\Pi$ consists of 
 genuine functions.  Let $u$ be a nilpotent element in $\frak{g}$ such that $\Pi_{N_u,\psi_u}\neq 0$. Let $\ul{p}$  be the partition corresponding to $u$.  Then $\Pi_{N_{u'},\psi_{u'}}\neq 0$ for a nilpotent element $u'$ whose partition is the special expansion $\ul{p}^G$  of  $\ul{p}$. 
 \end{thm} 

A more refined information about the conjugacy class of $u'$ is given by repeated application of Proposition \ref{P:forms}.

\section{Wave-front sets of exceptional groups}

We assume that $G$ is a split,  simply connected exceptional group. Let $(u,s,v)$ be an $\mathfrak{sl}_2$-triple in $\frak{g}$ where $s$ is a semi-simple element. 
The index of the orbit of $u$ is the value $\kappa(s,s)$ where the Killing form has been normalized so that the index of the orbit corresponding to the 
long root $\mathfrak{sl}_2$ is 1, as in \cite[Section 10]{Dyn52}.  For exceptional groups this invariant essentially determines the orbit over a separable closure. 
Let $C(s)$ be the centralizer of $s$ in $G$. It is a Levi subgroup of $G$. 
The centralizer $C$  in $G$ of the $\mathfrak{sl}_2$-triple coincides with the stabilizer in $C(s)$ of $v\in \frak{g}(2)$.  
The absolute type of $C$ is well known, however, a particular choice of $v$ defines a $k$-rational form of $C$. 

Let $L$ be the derived group of $C(s)$.  It is somewhat easier to work with $L$.  We determine 
 the stabilizer $S$ in $L$ of a generic point in $\frak{g}(2)$,  on a case by case basis, for all non-special orbits using the explicit structure of the $L$-module $\frak{g}(2)$ given in \cite{JN}.  Our computation works over most fields. Assuming that the characteristic of $k$ is not a bad prime for $G$ appears to be enough. 

Once we have $S$, we check whether the raising conditions in Propositions \ref{P5:local}  and \ref{P5:global}  (see also Corollaries \ref{cor:local} and \ref{cor:global}) 
are met for an $\SL_{2,\mathfrak c} \subseteq S$.  
If so, we enter the corresponding value of $m$ in the table.  In all but three cases (Sections  14.3, 17.15 and 17.23), $\SL_{2,\mathfrak c}$ is a long root $\SL_2 $ in $G$, so the conditions (1)-(3) are trivially to check and we do not include any additional explanation in these cases. The index of the raised orbit is  increased by 1.
 
If our method fails, we consider the method of M\oe glin. Now $k$ is a $p$-adic field.
M\oe glin's result states that if the orbit is not admissible, in the sense of Duflo, then it cannot be a leading term in a wave-front set. The question which orbits of $p$-adic groups are admissible has been studied by Nevins.
 According to Theorem 3.2 in \cite{N02}, the orbit of $u$ is not admissible if there exists $\SL_{2,\frak{c}}$ such that 
$\frak{g}(1)$, when decomposed as $\SL_{2,\frak{c}}$-module, satisfies the property that the total number of irreducible summands of dimensions $n\equiv 2\pmod{4}$ 
is odd. (In particular, 
all orbits that satisfy our raising conditions are not admissible.) If only the method of M\oe glin  applies, we write $\ast$. If both fail, we write $\ast\ast$. 
  These are precisely completely odd, non-special orbits. In particular, we have proved that only completely odd orbits can be admissible, as conjectured by Nevins. 

 In the following  five sections we tabulate our data.
  Notation is mostly self-explanatory, for example, $V_n$ denotes the standard representation of $\SL_n$ or $\Sp_{n}$, if $n$ is even, or the irreducible $n$-dimensional representation of $\SL_2$.

\section{$\rm G_2$} \label{S:G_2}

\begin{center}
    \begin{tabular}{| c | c | c | c |}
    \hline
    \multicolumn{4}{|c|}{Non-Special Nilpotent Orbits in Type $\rm G_2$}\\
    \hline
    Label  & Diagram & $S$&  $m$ \\
     & $\circ \Rrightarrow \circ$  &  &\\ \hline
    $A_1$ & 1 \ \ \ 0 & $ \SL_2 $ & $\ast\ast$   \\ \hline
    $\wt{A}_1$ & 0 \ \ \ 1 & $ \SL_2 $ &  1\\
    \hline
    \end{tabular}
\end{center}

\subsection{$A_1$}

\begin{align*}
L&=\SL_2,\\
\frak{g}(1)&=\Sym^3 V_2=V_4,\\
\frak{g}(2)&=V_1.
\end{align*}

Here $S=L$. Neither method works.

\subsection{$\wt{A}_1$}

\begin{align*}
L&=\SL_2,\\
\frak{g}(1)&=V_2,\\
\frak{g}(2)&=V_1.
\end{align*}

Here $S=L$, and $m=1$.

\section{$\rm F_4$}

\begin{center}
    \begin{tabular}{| c | c | c | c | }
    \hline
    \multicolumn{4}{|c|}{Non-Special Nilpotent Orbits in Type $\rm F_4$}\\
    \hline
    Label & Diagram & $S$ & $m$ \\
     & $\circ$ | $\circ$ $\Longrightarrow$ $\circ$ | $\circ$ &  &\\ \hline
    $A_1$ & 1 \ \ \ \ 0 \ \ \ \ \ \ 0 \ \ \ \ 0 & $\Sp_6$ &  5 \\ \hline
    $A_2 + \wt{A}_1$ & 0 \ \ \ \ 0 \ \ \ \ \ \ 1 \ \ \ \ 0 & $\SL_2$ &  $\ast$ \\ \hline
    $B_2$ &  2 \ \ \ \ 0 \ \ \ \ \ \ 0 \ \ \ \ 1 & $\SL_2(K) $ &  2\\ \hline
    $\wt{A}_2 + A_1$ &  0 \ \ \ \ 1 \ \ \ \ \ \ 0 \ \ \ \ 1  & $\SL_2$ & $\ast\ast$  \\ \hline
    $C_3(a_1)$ &  1 \ \ \ \ 0 \ \ \ \ \ \ 1 \ \ \ \ 0 & $\SL_2$ &  3\\ \hline
    \end{tabular}
\end{center}

\subsection{$A_1$}

\begin{align*}
L&=\Sp_6,\\
\frak{g}(1)&=\wedge^3(V_6)/V_6,\\
\frak{g}(2)&=V_1.
\end{align*}
Here $S=L$. Let $\SL_{2,\mathfrak c}$  correspond to a long root in $\Sp_6$, then $V_6=V_2 \oplus 4V_1$, as $\SL_{2,\mathfrak c}$-module. 
It follows that $\frak{g}(1) = 4V_1 \oplus 5V_2$, as $\SL_{2,\mathfrak c}$-module,  so $m=5$.

\subsection{$A_2+\wt{A_1}$}

\begin{align*}
L&=\SL_3 \times \SL_2,\\
\frak{g}(1)&=V_3 \otimes V_2,\\
\frak{g}(2)&=V_3^* \otimes \Sym^2 V_2.
\end{align*}
Here $\frak{g}(2)$ can be identified with the space of $3\times 3$ matrices so that determinant is a relative invariant. 
The stabilizer  $S$  of a generic point is $\SL_2$, diagonally embedded in $L$, where $\SL_2\rightarrow \SL_3$ is given by the natural action of 
$\SL_2$ on $\Sym^2 V_2$. 
Then $\frak{g}(1) = V_4 \oplus V_2$, as $\SL_{2,\mathfrak c}$-module, only the method of M\oe glin works.

\subsection{$B_2$}

\begin{align*}
L&=\Sp_4,\\
\frak{g}(1)&= V_4,\\
\frak{g}(2)&=V_1\oplus \wedge^2(V_4)/V_1.
\end{align*}
Here $V_5=\wedge^2(V_4)/V_1$ is the standard representation of $\SO_5\cong \Sp_4/\mu_2$. In particular, there is a degree $2$ relative invariant 
$\Delta$. The stabilizer $S$ of a generic point is $\SL_2(K)$ where $K=k(\sqrt{\Delta})$. Let  $\SL_{2,\mathfrak c}=\SL_2(k)\subseteq \SL_2(K)$.  Then the conditions 
(1)-(3) are satisfied with $m=2$. Since $2$ is even, Corollaries \ref{cor:local} and \ref{cor:global} do not apply i.e.  
 the Fourier-Jacobi model in  Propositions \ref{P5:local}  and \ref{P5:global}  is not a genuine representation of a 2-fold central extension of $\SL_{2, \mathfrak c}$. 
  However, as  $\SL_2(K)$-module, $\frak{g}(1)=V_2^K$ 
where $V_2^K$ is the standard 2-dimensional representation over $K$. In particular, the Fourier-Jacobi model in  Propositions \ref{P5:local}  and \ref{P5:global}  is a genuine representation of 2-fold central extension of  $\SL_2(K)$ and  hence it is Whittaker generic as $\SL_{2,\mathfrak c}$-module. Thus the orbit can be raised using  Propositions \ref{P5:local}  and \ref{P5:global}.

\subsection{$\wt{A_2}+A_1$}

\begin{align*}
L&=\SL_2^1 \times \SL_2^2,\\
\frak{g}(1)&=V_2^2 \oplus V_2^1 \otimes \Sym^2 V_2^2,\\
\frak{g}(2)&=V_1 \oplus V_2^1 \otimes V_2^2.
\end{align*}
Here $V_2^1 \otimes V_2^2$ can be identified with the space of $2\times 2$ matrices so that determinant is a relative invariant. 
Hence the stabilizer  $S$ of a generic point in $\frak{g}(2)$ is $\SL_2$,  diagonally embedded in $L$. 
Then $\frak{g}(1) = 2V_2 \oplus V_4$, 
as $\SL_{2,\mathfrak c}$-module. Neither method works.

\subsection{$C_3(a_1)$}

\begin{align*}
L&=\SL_2^1 \times \SL_2^2,\\
\frak{g}(1)&=V_2^1 \oplus V_2^1 \otimes V_2^2,\\
\frak{g}(2)&=V_2^2 \oplus \Sym^2V_2^2.
\end{align*}
The space $\Sym^2V_2^2$ has a degree 2 pseudo-invariant. Hence, the stabilizer in $\SL_2^2$ of a generic point here is $\SO_2$. Since the stabilizer in 
$\SO_2$ of a generic point in $V_2^2$ is trivial, it follows that the stabilizer $S$ of a generic point in $\frak{g}(2)$ is $\SL_2^1$. 
Hence, $\frak{g}(1) = 3V_2$, as $\SL_{2,\mathfrak c}$-module,  and $m=3$.

\section{$\rm E_6$}

\begin{center}
    \begin{tabular}{| c | c | c | c |}
    \hline
    \multicolumn{4}{|c|}{Non-Special Nilpotent Orbits in Type $\rm E_6$}\\
    \hline
    Label & Diagram & $S$ &  $m$\\
    & $\circ$ &  &\\
    & $|$ & &\\
     & $\circ$ | $\circ$ | $\circ$ | $\circ$ | $\circ$   & &\\ \hline
      & 0 &  &\\
    $3A_1$ & 0 \ \ \ \ 0 \ \ \ \ 1 \ \ \ \ 0 \ \ \ \ 0 & $\SL_3\times \SL_2$ &  9 \\ \hline
          & 0 &  &\\
    $2A_2 + {A}_1$ & 1 \ \ \ \ 0 \ \ \ \ 1 \ \ \ \ 0 \ \ \ \ 1 & $\SL_2$ & $\ast\ast$   \\ \hline
          & 1 &  &\\
    $A_3 + A_1$ & 0 \ \ \ \ 1 \ \ \ \ 0 \ \ \ \ 1 \ \ \ \ 0  & $\SL_2 $ &  5\\ \hline
          & 1 &  &\\
    $A_5$ & 2 \ \ \ \ 1 \ \ \ \ 0 \ \ \ \ 1 \ \ \ \ 2  & $\SL_2$ &  3\\ \hline
    \end{tabular}
\end{center}

\subsection{$3A_1$}

\begin{align*}
L&=\SL_3^1 \times \SL_2 \times \SL_3^2,\\
\frak{g}(1)&=V_3^1 \otimes V_2 \otimes V_3^{2,*},\\
\frak{g}(2)&=V_3^{1,*} \otimes V_3^2.
\end{align*}
The space $V_3^{1,*} \otimes V_3^2$ can be identified with the spaces of $3\times 3$ matrices, so that determinant is a relative invariant. Hence the stabilizer  $S$ in $L$ of a generic point 
in $\frak{g}(2)$ is $\SL_3\times \SL_2$ where $\SL_3$ is diagonally embedded in $\SL_3^1 \times \SL_3^2$. Let $\SL_{2,\mathfrak c}$ be the second factor of $S$. 
Then $\frak{g}(1) = 9V_2$,  as $\SL_{2,\mathfrak c}$-module, so  $m=9$.

\subsection{$2A_2+A_1$}\label{ss:2A_2+A_1,E_6} 

\begin{align*}
L&=\SL_2^1 \times \SL_2^2 \times \SL_2^3,\\
\frak{g}(1)&=V_2^1 \oplus V_2^3 \oplus V_2^1 \otimes V_2^2 \otimes V_2^3,\\
\frak{g}(2)&=V_1 \oplus V_2^1 \otimes V_2^2\oplus V_2^2 \otimes V_2^3 .
\end{align*}
The spaces $V_2^1 \otimes V_2^2$ and $V_2^2 \otimes V_2^3$  can be identified with the spaces of $2\times 2$ matrices. Hence $S=\SL_2$, embedded 
diagonally into the three $\SL_2$, and $\frak{g}(1) =  4V_2 \oplus V_4$.  Neither method works.

\subsection{$A_3 + A_1$}

\begin{align*}
L&=\SL_2^1 \times \SL_2^2 \times \SL_2^3,\\
\frak{g}(1)&=V_2^2 \oplus V_2^1 \otimes V_2^2 \oplus V_2^2 \otimes V_2^3,\\
\frak{g}(2)&=V_2^1 \oplus V_2^3 \oplus V_2^1 \otimes V_2^3.
\end{align*}
The stabilizer in $\SL_2^1 \times \SL_2^3$ of a generic point in $V_2^1\otimes V_2^3$ is $\SL_2$ diagonally embedded. The stabilizer in 
this $\SL_2$ of a generic point in $V_2^1\oplus V_2^3$ is trivial. Hence $S=\SL_2^2$, and $\frak{g}(1) = 5V_2$, so $m=5$.

\subsection{$A_5$}

\begin{align*}
L&=\SL_2,\\
\frak{g}(1)&=V_2 \oplus V_2 \oplus V_2,\\
\frak{g}(2)&=5V_1.
\end{align*}
Here $L=S$. Hence $\frak{g}(1) = 3V_2$, so $m=3$.

\section{$\rm E_7$}

\begin{center}
    \begin{tabular}{| c | c | c | c  |}
    \hline
    \multicolumn{4}{|c|}{Non-Special Nilpotent Orbits in Type $\rm E_7$}\\
    \hline
    Label & Diagram & $S$ &  $m$\\
    & $\ \ \ \ \ \ \circ$ &  &\\
    & $\ \ \ \ \ \ |$ &  &\\
     & $\circ$ | $\circ$ | $\circ$ | $\circ$ | $\circ$ | $\circ$ & &\\ \hline
      & \ \ \ \ \ \ 0 &  &\\
    $(3A_1)'$ & 0 \ \ \ \ 0 \ \ \ \ 0 \ \ \ \ 0 \ \ \ \ 1 \ \ \ \ 0 & $\Sp_6 \times \SL_2$ &  15 \\ \hline
         & \ \ \ \ \ \ 1 &  &\\
    $4{A}_1$ &  1 \ \ \ \ 0 \ \ \ \ 0 \ \ \ \ 0 \ \ \ \ 0 \ \ \ \ 0  & $\Sp_6$ &  7 \\ \hline
         & \ \ \ \ \ \ 0 & &\\
    $2A_2 + A_1$ &  0 \ \ \ \ 1 \ \ \ \ 0 \ \ \ \ 0 \ \ \ \ 1 \ \ \ \ 0   & $\SL_2\times \SL_2$ &  $\ast\ast$  \\ \hline
              & \ \ \ \ \ \ 0 &  &\\
    $(A_3+A_1)'$ &  0 \ \ \ \ 0 \ \ \ \ 0 \ \ \ \ 1 \ \ \ \ 0 \ \ \ \ 1 & $\SL_2\times \SL_2\times\SL_2$ &  9 \\ \hline
          & \ \ \ \ \ \ 0 &  &\\
    $A_3+2A_1$ &  1 \ \ \ \ 0 \ \ \ \ 1 \ \ \ \ 0 \ \ \ \ 0 \ \ \ \ 1 & $\SL_2\times \SL_2$ &  5 \\ \hline
    & \ \ \ \ \ \ 1 &  &\\
    $D_4+A_1$ &  1 \ \ \ \ 0 \ \ \ \ 0 \ \ \ \ 0 \ \ \ \ 1 \ \ \ \ 2 & $\Sp_4$ & 3 \\ \hline

         & \ \ \ \ \ \ 0 &  &\\
    $(A_5)'$ &  0 \ \ \ \ 2 \ \ \ \ 0 \ \ \ \ 1 \ \ \ \ 0 \ \ \ \ 1 & $\SL_2\times \SL_2$ &  5\\ \hline
      & \ \ \ \ \ \ 0 &  &\\
    $A_5+A_1$ & 2 \ \ \ \ 1 \ \ \ \ 0 \ \ \ \ 1 \ \ \ \ 0 \ \ \ \ 1 & $\SL_2$ &  $\ast\ast$ \\ \hline
         & \ \ \ \ \ \ 1 &  &\\
    $D_6(a_2)$ &  2 \ \ \ \ 0 \ \ \ \ 1 \ \ \ \ 0 \ \ \ \ 1 \ \ \ \ 0  & $\SL_2$ &  5 \\ \hline
             & \ \ \ \ \ \ 1 &  &\\
    $D_6$ &  2 \ \ \ \ 2 \ \ \ \ 1 \ \ \ \ 0 \ \ \ \ 1 \ \ \ \ 2   & $\SL_2$ &  3\\ \hline
    \end{tabular}
\end{center}

\subsection{$(3A_1)'$}

\begin{align*}
L&=\SL_6 \times \SL_2,\\
\frak{g}(1)&=\wedge^2 V_6 \otimes V_2,\\
\frak{g}(2)&=\wedge^2 V_6^*.
\end{align*}
The stabilizer $S$ of generic point in $\frak{g}(2)$
is $\Sp_6 \times \SL_2$. We let $\SL_{2,\mathfrak c}$ be the second factor of $S$. Then $\frak{g}(1) = 15V_2$, as $\SL_{2,\mathfrak c}$-module, so $m=15$.

\subsection{$4A_1$} \label{ss:4A_1}
 
\begin{align*}
L&=\SL_6,\\
\frak{g}(1)&=V_6^* \oplus \wedge^3V_6,\\
\frak{g}(2)&=V_1 \oplus \wedge^2 V_6.
\end{align*}
The stabilizer $S$ of generic point in $\frak{g}(2)$
is $\Sp_6$. Let $\SL_2^c$ be a long root $\SL_2$.  
Then $V_6 = V_2 \oplus 4V_1$, as $\SL_{2,\mathfrak c}$-module, and  $\frak{g}(1) = 12 V_1 \oplus 7V_2$, so $m=7$.

\subsection{$2A_2+A_1$}\label{ss:2A_2+A_1} 

\begin{align*}
L&=\SL_2^1 \times \SL_4 \times \SL_2^2,\\
\frak{g}(1)&=V_2^1 \otimes V_4^* \oplus \wedge^2 V_4 \otimes V_2^2,\\
\frak{g}(2)&=V_1 \oplus V_2^1 \otimes V_4 \otimes V_2^2.
\end{align*}
The group $\SL_2^1\times\SL_2^2$ acting on $V_2^1  \otimes V_2^2$ gives an identification of $(\SL_2^1\times\SL_2^2)/\mu_2\cong \SO_4$. Hence the second 
summand of $\frak{g}(2)$ can be identified with the space of $4\times 4$ matrices. In this way the determinant is a relative nvariant, and the stabilizer $S$ of a 
generic point of $\frak{g}(2)$ is $\SL_2\times\SL_2$ embedded diagonally into $L$.  It is the tensor product embedding into $\SL_4$, so $V_4=V_2^1\otimes V_2^2$, 
under the restriction.

(Case 1) Let $\SL_{2,\mathfrak c}$ be the factor of $S$ isomorphic to $\SL_2^1$. Then $\frak{g}(1)=8V_1  \oplus 4V_3$, as $\SL_{2,\mathfrak c}$-module, 
hence neither method works.

(Case 2) Let $\SL_{2,\mathfrak c}$  be the factor of $S$ isomorphic to $\SL_2^2$. Then  $\frak{g}(1)=8V_2  \oplus V_4$, as $\SL_{2,\mathfrak c}$-module, 
hence neither method works.

\subsection{$(A_3+A_1)'$}

\begin{align*}
L&=\SL_4 \times \SL_2^1 \times \SL_2^2,\\
\frak{g}(1)&=V_2^2 \oplus V_4 \otimes V_2^1 \otimes V_2^2,\\
\frak{g}(2)&=V_4\otimes V_2^1  \oplus \wedge^2 V_4.
\end{align*}
Since $\SL_2^2$ acts trivially on $\frak{g}(2)$, it is a factor of the stabilizer $S$. The stabilizer in $\SL_4 \times \SL_2^1$ of a generic point in 
$\frak{g}(2)$ is $\SL_2 \times \SL_2$. This will be discussed in the next case, as it is not needed here. Let $\SL_{2,\mathfrak c}=\SL_2^2$. 
Then $\frak{g}(1) = 9V_2$, so $m=9$.

\subsection{$A_3+2A_1$}\label{ss:A_3+A_1} 

\begin{align*}
L&=\SL_2 \times \SL_4,\\
\frak{g}(1)&=V_2 \oplus V_4 \oplus V_2\otimes  \wedge^2 V_4,\\
\frak{g}(2)&=V_1 \oplus \wedge^2 V_4 \oplus V_2 \otimes V_4^*.
\end{align*}
A generic element in $V_2\otimes V_4^*$ is contained in $V_2\otimes V_2^1$ where $V_2^1\subset V_4^*$ is a 2-dimensional subspace.
A generic element in $\wedge^2 V_4$ is a non-degenerate 
symplectic form $\omega$ on $V^*_4$. Let $\Sp_4\subset \SL_4$ be the stabilizer of $\omega$. Now we have two cases. The restriction of $\omega$ to $V_2^1$ is 
trivial or the restriction of $\omega$ to $V_2^1$ is non-trivial, hence non-degenerate as $\omega$ is a symplectic form. This is the generic case. In this case we can write 
$V_4^*=V_2^1 \oplus V_2^2$ where $V_2^2$ is the orthogonal complement of $V_2^1$. Now we have corresponding inclusions 
$$
\SL_2^1 \times \SL_2^2 \subset \Sp_4 \subset \SL_4. 
$$
Hence the stabilizer in $\SL_2 \times \Sp_4$ of a generic point in $V_2\otimes V_2^1$ is $\SL_2^3\times \SL_2^2$ where $\SL_2^3$ is diagonally embedded in 
$\SL_2\times \SL_2^1$. This is the stabilizer $S$ of a generic point in $\frak{g}(2)$. Let $\SL_{2,\mathfrak c}=\SL_2^2$. Then, as $\SL_{2,\mathfrak c}$-module, $\frak{g}(1)=\ 8 V_1 \oplus 5 V_2$, 
so $m=5$.

\subsection{$D_4+A_1$} \label{ss:D_4+A_1}. 

\begin{align*}
L&=\SL_4,\\
\frak{g}(1)&=2V_4 \oplus V_4^*,\\
\frak{g}(2)&=3V_1 \oplus \wedge^2 V_4.
\end{align*}
The stabilizer $S$  of a generic point in $\frak{g}(2)$ is $\Sp_4$. Let $\SL_{2,\mathfrak c}$ be the long root $\SL_2$ in $\Sp_4$. Then 
$V_4^*  = V_2 \oplus 2V_1$, as $\SL_{2,\mathfrak c}$-module.  Hence $\frak{g}(1) = 3V_2 \oplus 6 V_1$, so  $m=3$.

\subsection{$(A_5)'$}

\begin{align*}
L&=\SL_2^1 \times \SL_2^2 \times \SL_2^3 \times\SL_2^4,\\
\frak{g}(1)&=V_2^4 \oplus V_2^2 \otimes V_2^3 \otimes V_2^4,\\
\frak{g}(2)&=V_1\oplus V_2^1 \otimes V_2^2 \oplus V_2^2 \otimes V_2^3. 
\end{align*}
The stabilizer $S$ is $\SL_2\times \SL_2^4$ where $\SL_2$ is diagonally embedded into the first three $\SL_2$'s. Let $\SL_{2,\mathfrak c}=\SL_2^4$. 
Then, as $\SL_{2,\mathfrak c}$-module, $\frak{g}(1) = 5V_2^4$,  so $m=5$.

\subsection{$A_5+A_1$}
 
\begin{align*}
L&=\SL_2^1 \times \SL_2^2 \times \SL_2^3,\\
\frak{g}(1)&=V_2^1 \oplus V_2^3 \oplus  V_2^1 \otimes V_2^2 \otimes V_2^3,\\
\frak{g}(2)&=2V_1 \oplus V_2^1 \otimes V_2^2 \oplus V_2^2 \otimes V_2^3.
\end{align*}
The same as Section \ref{ss:2A_2+A_1,E_6}. 

\subsection{$D_6(a_2)$} \label{ss:D_6(a_2)}
 
\begin{align*}
L&=\SL_2^1 \times \SL_2^2 \times \SL_2^3,\\
\frak{g}(1)&=V_2^2 \oplus V_2^1 \otimes V_2^2 \oplus V_2^2 \otimes V_2^3 ,\\
\frak{g}(2)&=2V_2^1 \oplus V_2^3 \oplus V_2^1 \otimes V_2^3.
\end{align*}
The stabilizer in $\SL_2^1\times \SL_2^3$ of a generic point in $V_2^1 \otimes V_2^3$ is $\SL_2$ embedded diagonally. The stabilizer in 
this $\SL_2$ of a generic point in $ V_2^1 \oplus V_2^3$ is trivial. Hence $S=\SL_2^2$. Then, as $\SL_{2,\mathfrak c}$-module,  $\frak{g}(1) = 5V_2$, so $m=5$.

\subsection{$D_6$}
 
\begin{align*}
L&=\SL_2,\\
\frak{g}(1)&=V_2 \oplus V_2 \oplus V_2,\\
\frak{g}(2)&=6V_1.
\end{align*}
Here $S=L$. Hence $\frak{g}(1) = 3V_2$, so $m=3$.

\section{$\rm E_8$} \label{S:E_8}

\begin{center}
    \begin{tabular}{| c | c | c | c |}
    \hline
    \multicolumn{4}{|c|}{Non-Special Nilpotent Orbits in Type $\rm E_8$}\\
    \hline
    Label & Diagram & $S$  & $m$\\
    & $\ \ \ \ \ \ \ \ \ \ \ \ \circ$ & & \\
    & $\ \ \ \ \ \ \ \ \ \ \ \ |$ & & \\
     & $\circ$ | $\circ$ | $\circ$ | $\circ$ | $\circ$ | $\circ$ | $\circ$ & & \\ \hline
      & \ \ \ \ \ \ \ \ \ \ \ \ \ 0 & & \\
    $3A_1$ & 0 \ \ \ \ 1 \ \ \ \ 0 \ \ \ \ 0 \ \ \ \ 0 \ \ \ \ 0 \ \ \ \ 0 & $\mathrm{F}_4 \times \SL_2$ & 27 \\ \hline
      & \ \ \ \ \ \ \ \ \ \ \ \ \ 1 &  &\\
    $4{A}_1$ &0 \ \ \ \ 0 \ \ \ \ 0 \ \ \ \ 0 \ \ \ \ 0 \ \ \ \ 0 \ \ \ \ 0 & $\Sp_8$ &  15 \\ \hline
      & \ \ \ \ \ \ \ \ \ \ \ \ \ 0 &  &\\
    $A_2 + 3A_1$ &  0 \ \ \ \ 0 \ \ \ \ 0 \ \ \ \ 0 \ \ \ \ 0 \ \ \ \ 1 \ \ \ \ 0  & $\mathrm{G}_2 \times \SL_2$ &  21 \\ \hline
      & \ \ \ \ \ \ \ \ \ \ \ \ \ 0 &  &\\
    $2A_2+A_1$ & 0 \ \ \ \ 1 \ \ \ \ 0 \ \ \ \ 0 \ \ \ \ 0 \ \ \ \ 0 \ \ \ \ 1 & $\mathrm{G}_2\times \SL_2$ & $\ast\ast$ \\ \hline
      & \ \ \ \ \ \ \ \ \ \ \ \ \ 0 &  &\\
    $A_3+A_1$ &  1 \ \ \ \ 0 \ \ \ \ 1 \ \ \ \ 0 \ \ \ \ 0 \ \ \ \ 0 \ \ \ \ 0 & $B_3\times \SL_2 $ &  17 \\ \hline
        & \ \ \ \ \ \ \ \ \ \ \ \ \ 0 &  &\\
    $2A_2+2A_1$ &  0 \ \ \ \ 0 \ \ \ \ 0 \ \ \ \ 1 \ \ \ \ 0 \ \ \ \ 0 \ \ \ \ 0 & $\Sp_4$ & $\ast\ast$ \\ \hline
     \end{tabular}
\end{center}

\begin{center}
    \begin{tabular}{| c | c | c | c | }
    \hline
    \multicolumn{4}{|c|}{Non-Special Nilpotent Orbits in Type $\rm E_8$(continued)}\\
    \hline
    Label & Diagram & $S$ & $m$\\
    & $\ \ \ \ \ \ \ \ \ \ \ \ \circ$ &  &\\
    & $\ \ \ \ \ \ \ \ \ \ \ \ |$ &  &\\
     & $\circ$ | $\circ$ | $\circ$ | $\circ$ | $\circ$ | $\circ$ | $\circ$ &  &\\ \hline
    & \ \ \ \ \ \ \ \ \ \ \ \ \ 0 &  &\\
    $A_3+2A_1$ &  1 \ \ \ \ 0 \ \ \ \ 0 \ \ \ \ 0 \ \ \ \ 0 \ \ \ \ 1 \ \ \ \ 0 & $\Sp_4 \times \SL_2$ &  9 \\ \hline    
    & \ \ \ \ \ \ \ \ \ \ \ \ \ 0 &  &\\
    $A_3+A_2+A_1$ &  0 \ \ \ \ 0 \ \ \ \ 0 \ \ \ \ 0 \ \ \ \ 1 \ \ \ \ 0 \ \ \ \ 0 & $\SL_2\times A_1$ &  15 \\ \hline
    & \ \ \ \ \ \ \ \ \ \ \ \ \ 1 &  &\\
    $D_4+A_1$ &  2 \ \ \ \ 1 \ \ \ \ 0 \ \ \ \ 0 \ \ \ \ 0 \ \ \ \ 0 \ \ \ \ 0 & $\Sp_6$ &  7 \\ \hline
    & \ \ \ \ \ \ \ \ \ \ \ \ \ 0 &  &\\
    $2A_3$ &  0 \ \ \ \ 0 \ \ \ \ 0 \ \ \ \ 1 \ \ \ \ 0 \ \ \ \ 0 \ \ \ \ 1 & $\Sp_4$ & $\ast$ \\ \hline
          & \ \ \ \ \ \ \ \ \ \ \ \ \ 0 &  &\\
    $A_5$ & 1 \ \ \ \ 0 \ \ \ \ 1 \ \ \ \ 0 \ \ \ \ 0 \ \ \ \ 0 \ \ \ \ 2 & $\mathrm{G}_2\times \SL_2$ &  9 \\ \hline
      & \ \ \ \ \ \ \ \ \ \ \ \ \ 0 &  &\\
    $A_4+A_3$ &0 \ \ \ \ 1 \ \ \ \ 0 \ \ \ \ 0 \ \ \ \ 1 \ \ \ \ 0 \ \ \ \ 0 & $\SL_2$ & $\ast\ast$ \\ \hline
    & \ \ \ \ \ \ \ \ \ \ \ \ \ 0 &  &\\
    $A_5+A_1$ &  1 \ \ \ \ 0 \ \ \ \ 0 \ \ \ \ 0 \ \ \ \ 1 \ \ \ \ 0 \ \ \ \ 1  & $\SL_2 \times \SL_2$ &  5 \\ \hline
      & \ \ \ \ \ \ \ \ \ \ \ \ \ 0 &  &\\
    $D_5(a_1)+A_2$ & 1 \ \ \ \ 0 \ \ \ \ 1 \ \ \ \ 0 \ \ \ \ 0 \ \ \ \ 1 \ \ \ \ 0 & $\SL_2$ & $\ast$ \\ \hline
      & \ \ \ \ \ \ \ \ \ \ \ \ \ 1 &  &\\
    $D_6(a_2)$ &  0 \ \ \ \ 1 \ \ \ \ 0 \ \ \ \ 0 \ \ \ \ 0 \ \ \ \ 1 \ \ \ \ 0 & $\SL_2(K)$ &  10 \\ \hline
    & \ \ \ \ \ \ \ \ \ \ \ \ \ 0 &  &\\
    $E_6(a_3)+A_1$ &  0 \ \ \ \ 1 \ \ \ \ 0 \ \ \ \ 1 \ \ \ \ 0 \ \ \ \ 0 \ \ \ \ 1 & $\SL_2$ & $\ast\ast$ \\ \hline
    & \ \ \ \ \ \ \ \ \ \ \ \ \ 0 &  &\\
    $E_7(a_5)$ &  0 \ \ \ \ 0 \ \ \ \ 1 \ \ \ \ 0 \ \ \ \ 1 \ \ \ \ 0 \ \ \ \ 0 & $\SL_2\times \Aut^1(E)$ &  9 \\ \hline
    & \ \ \ \ \ \ \ \ \ \ \ \ \ 0 &  &\\
    $D_5+A_1$ &  2 \ \ \ \ 1 \ \ \ \ 0 \ \ \ \ 1 \ \ \ \ 0 \ \ \ \ 0 \ \ \ \ 1 & $\SL_2 \times \SL_2$ &  5 \\ \hline
    & \ \ \ \ \ \ \ \ \ \ \ \ \ 1 &  &\\
    $D_6$ &  2 \ \ \ \ 1 \ \ \ \ 0 \ \ \ \ 0 \ \ \ \ 0 \ \ \ \ 1 \ \ \ \ 2 & $\Sp_4$ &  3 \\ \hline
    & \ \ \ \ \ \ \ \ \ \ \ \ \ 0 &  &\\
    $A_7$ &  0 \ \ \ \ 1 \ \ \ \ 1 \ \ \ \ 0 \ \ \ \ 1 \ \ \ \ 0 \ \ \ \ 1 & $\SL_2$ & $\ast$ \\ \hline
        & \ \ \ \ \ \ \ \ \ \ \ \ \ 0 &  &\\
    $E_6+A_1$ &  2 \ \ \ \ 2 \ \ \ \ 1 \ \ \ \ 0 \ \ \ \ 1 \ \ \ \ 0 \ \ \ \ 1 & $\SL_2$ & $\ast\ast$ \\ \hline
        & \ \ \ \ \ \ \ \ \ \ \ \ \ 1 &  &\\
    $E_7(a_2)$ &  2 \ \ \ \ 2 \ \ \ \ 0 \ \ \ \ 1 \ \ \ \ 0 \ \ \ \ 1 \ \ \ \ 0 & $\SL_2$ &  5 \\ \hline
        & \ \ \ \ \ \ \ \ \ \ \ \ \ 1 &  &\\
    $D_7$ &  1 \ \ \ \ 0 \ \ \ \ 1 \ \ \ \ 1 \ \ \ \ 0 \ \ \ \ 1 \ \ \ \ 2 & $\SL_2$ &  5 \\ \hline
        & \ \ \ \ \ \ \ \ \ \ \ \ \ 1 &  &\\
    $E_7$ &  2 \ \ \ \ 2 \ \ \ \ 2 \ \ \ \ 1 \ \ \ \ 0 \ \ \ \ 1 \ \ \ \ 2 & $\SL_2$ &  3\\ \hline
    \end{tabular}
\end{center}

\subsection{$3A_1$}

\begin{align*}
L&=\SL_2 \times E_6,\\
\frak{g}(1)&=V_2 \otimes V_{27}^*,\\
\frak{g}(2)&=V_{27}.
\end{align*}
The generic stabilizer $S$ is $\SL_2\times F_4$. Let $\SL_{2,\mathfrak c}=\SL_2$, the first factor. 
Then $\frak{g}(1) = 27V_2$, so $m=27$.

\subsection{$4A_1$}

\begin{align*}
L&=\SL_8,\\
\frak{g}(1)&=\wedge^3 V_8,\\
\frak{g}(2)&=\wedge^2 V_8^*.
\end{align*} The stabilizer $S$ is $\Sp_8$. Let $\SL_{2,\mathfrak c}$ be a long root $\SL_2$ in $\Sp_6$. Then 
 $V_8 = V_2 \oplus 6 V_1$, as $\SL_{2,\mathfrak c}$-module. Hence  $\frak{g}(1)= 26 V_1 \oplus 15 V_2$,  so $m=15$.

\subsection{$A_2 + 3 A_1$}

\begin{align*}
L&=\SL_7 \times \SL_2,\\
\frak{g}(1)&= \wedge^2 V_7\otimes V_2,\\
\frak{g}(2)&=\wedge^3 V_7^*.
\end{align*}
The stabilizer in $\SL_7$ of a generic point in $\wedge^3 V_7^*$ is $G_2$. Hence the stabilizer $S$ is $G_2 \times \SL_2$.  
Let $\SL_{2,\mathfrak c}=\SL_2$, the second factor in $S$. 
Then, as $\SL_{2,\mathfrak c}$-module, $\frak{g}(1)=21 V_2$, so $m=21$.

\subsection{$2A_2 + A_1$}

\begin{align*}
L &=\SL_2 \times \Spin_{10},\\
\frak{g}(1)&=\ V_2 \otimes V_{10} \oplus V_{16}^2,\\
\frak{g}(2)&=V_1 \oplus V_2 \otimes V_{16}^1.
\end{align*}
Here $V_{16}^i$ ($i=1,2$) denote two inequivalent spin representations.  A generic element in 
$V_2 \otimes V_{16}^1$ is contained in $V_2\otimes V_2^1$ for some 2-dimensional subspace $V_2^1\subset V_{16}^1$. 
The point-wise stabilizer in $\Spin_{10}$ of a generic $V_2^1$ is $G_2$. The centralizer of $G_2$ in $\Spin_{10}$ is 
$\Spin_3$. Since $\Spin_3$ acts on $V_2^1$, we have an identification $\Spin_3\cong \SL_2^1$. It follows that the stabilizer $S$ of 
generic point in $\frak{g}(2)$ is $\SL_2\times G_2$ where $\SL_2$ is embedded diagonally. 

(Case 1) Let $\SL_{2,\mathfrak c}$ be the first factor of $S$. Then $V_{10}=V_3 \oplus 7V_1$ and $V_{16}^2=8V_2$ as $\SL_{2,\mathfrak c}$-modules. 
Then $\frak{g}(1)= 16V_2 \oplus V_4$, hence neither method works.

(Case 2) Take the $\SL_{2,\mathfrak c}$ corresponding to a long root $\SL_2$ in $G_2$. It is a root $\SL_2$  in $\Spin_{10}$. 
Then  $V_{10}= 2V_2 \oplus 6V_1$ and $V_{16}=4 V_3 \oplus 8 V_1$, as $\SL_{2,\mathfrak c}$-modules, and 
  $\frak{g}(1)= 20 V_1 \oplus  8 V_2$. Hence neither method works.

\subsection{$A_3 + A_1$}

\begin{align*}
L&=\SL_2 \times \Spin_{10},\\
\frak{g}(1)&=V_2 \oplus  V_2\otimes V_{16},\\
\frak{g}(2)&=V_{16} \oplus V_{10}.
\end{align*}
Here $V_{16}$ denotes a spin representation. Note that $\SL_2$, the first factor of $L$, is always in 
$S$. So we let $\SL_{2,\mathfrak c}$ be this $\SL_2$. 
Then $\frak{g}(1)=17 V_2$, as $\SL_{2,\mathfrak c}$-module, so $m=17$.

\subsection{$2A_2 + 2A_1$}

\begin{align*}
L&=\SL_4 \times \SL_5,\\
\frak{g}(1)&=V_4 \otimes  \wedge^2 V_5^*,\\
\frak{g}(2)&= \wedge^2 V_4 \otimes V_5.
\end{align*}
Write $V_6=\wedge^2 V_4$. The action of $\SL_4$ on $V_6$ gives an isomorphism $\SL_4/\mu_2\cong \SO_6$. A generic element in $V_6\otimes V_5$ is 
contained in $V_5^1\otimes V_5$ for a 5-dimensional non-degenerate subspace $V_5^1\subset V_6$. 
Note that $V_5^1$ is a split quadratic space, as $V_6$ is. The stabilizer $S$ of 
a generic point is $\Sp_4\cong \Spin(V_5^1)$ embedded diagonally in $\SL_4\times \SL_5$. Let 
$\SL_{2,\mathfrak c}$ be a long root $\SL_2$ in $\Sp_4$. Then $V_4 = V_2 \oplus 2V_1$, and $V_5 = 2V_2 \oplus V_1$, as $\SL_{2,\mathfrak c}$-modules, and  
$\frak{g}(1)= 8V_1 \oplus 8 V_2 \oplus 4 V_3 \oplus V_4$. Hence  neither method works.

\subsection{$A_3 + 2A_1$} 

\begin{align*}
L&=\SL_6 \times \SL_2,\\
\frak{g}(1)&=V_6^* \oplus \wedge^2 V_6 \otimes V_2,  \\
\frak{g}(2)&=  V_6\otimes V_2 \oplus \wedge^2 V_6^*.
\end{align*}
A generic element in $\wedge^2 V_6^*$ is a non-degenerate symplectic form $\omega$ on $V_6$.  Let $\Sp_6$ be the stabilizer of 
$\omega$ in $\SL_6$. A generic element in 
$V_6\otimes V_2$ is contained in $V_2^1\otimes V_2$ where $V_2^1$ is a 2-dimensional subspace of $V_6$. We are in a generic 
situation when $\omega$ restricts to a non-trivial form on $V_2^1$. In this case we can decompose $V_6=V_2^1\oplus V_4$. This gives an 
embedding $\SL_2^1\times \Sp_4 \subset \Sp_6$. The centralizer $S$ of a generic element in $\frak{g}(2)$ is $\Sp_4\times \SL_2$ 
where $\SL_2$ is diagonally embedded in $\SL_2^1$ and the second factor of $L$. Let $\SL_{2,\mathfrak c}$ be a long root $\SL_2$ in 
$\Sp_4$. Then  $V_6^* \cong V_6 = V_2 \oplus 4V_1$, as $\SL_{2,\mathfrak c}$-modules, and $\frak{g}(1)=18 V_1 \oplus 9 V_2$, 
so  $m=9$.

\subsection{$A_3 + A_2 + A_1$}

\begin{align*}
L&=\SL_5 \times \SL_2 \times \SL_3,\\
\frak{g}(1)&= V_5 \otimes V_2 \otimes V_3^*,\\
\frak{g}(2)&=\wedge^2 V_5\otimes V_3  .
\end{align*}
We take $\SL_{2,\mathfrak c}$ to be the second factor of $L$, it is clearly contained in $S$. 
Then $\frak{g}(1)=15V_2$,  as $\SL_{2,\mathfrak c}$-module, so $m=15$. 

\subsection{$D_4 + A_1$}

\begin{align*}
L&=\SL_6,\\
\frak{g}(1)&=V_6^* \oplus \wedge^3 V_6,\\
\frak{g}(2)&=2V_1 \oplus \wedge^2 V_6.
\end{align*}
The same as Section \ref{ss:4A_1}.

\subsection{$2A_3$}

\begin{align*}
L&=\SL_4^1 \times \SL_4^2,\\
\frak{g}(1)&=V_4^2 \oplus V_4^1 \otimes \wedge^2 V_4^2,\\
\frak{g}(2)&=V_4^1 \otimes V_4^{2,*} \oplus \wedge^2 V_4^1.
\end{align*}
The stabilizer of a generic point in $V_4^1\otimes V_4^{2,*}$ is $\SL_4$ embedded diagonally. The stabilizer of a generic point in 
$ \wedge^2 V_4^1$ is $\Sp_4 \times \SL_4^2$. Hence $S=\Sp_4$ embedded diagonally. Let $\SL_{2,\mathfrak c}$ be a long root $\SL_2$ in $\Sp_4$. Then, 
as $\SL_2^c$-modules, $V_4^1=V_4^2=V_2 + 2V_1$. Hence $\frak{g}(1)=  8V_1 \oplus 7V_2 \oplus 2V_3$ and the method of M\oe glin works.

\subsection{$A_5$}  
\begin{align*}
L&=  \SL_2 \times \Spin_8,\\
\frak{g}(1)&=V_2 \oplus V_2 \otimes V_8,\\
\frak{g}(2)&=V_8 \oplus V_1 \oplus V_8'.
\end{align*}
Here we can proclaim $V_8$ in $\frak{g}(1)$ be the standard representation, then $V_8'$ is a spin representation.
 Note that $\SL_2$, the first factor of $L$, is always in $S$. 
 So we let $\SL_{2,\mathfrak c}$ be this $\SL_2$. 
Then $\frak{g}(1)=9 V_2$, as $\SL_{2,\mathfrak c}$-module, so $m=9$.

\subsection{$A_4 + A_3$}

\begin{align*}
L &=\SL_2^1 \times \SL_3^1 \times \SL_2^2 \times  \SL_3^2,\\
\frak{g}(1)&=V_2^1 \otimes V_3^{1,*} \oplus V_3^{1} \otimes V_2^2 \otimes V_3^{2,*},\\
\frak{g}(2)&=V_2^1 \otimes V_2^2 \otimes V_3^{2,*}  \oplus V_3^{1,*} \otimes V_3^2.
\end{align*}
The stabilizer is $\SL_2$, diagonally embedded into all factors of $L$, where we use the symmetric square embedding into 
two $\SL_3$'s. Then $\frak{g}(1)= 3 V_2 \oplus 3V_4 \oplus V_6$ and neither method works. 

\subsection{$A_5 + A_1$}

\begin{align*}
L&=\SL_4 \times \SL_2^1 \times \SL_2^2,\\
\frak{g}(1)&=V_2^2 \oplus V_4^* \oplus V_4 \otimes V_2^1 \otimes V_2^2,\\
\frak{g}(2)&=V_4  \otimes V_2^1 \oplus V_2^1 \otimes V_2^2 \oplus \wedge^2 V_4.
\end{align*}
The stabilizer of a generic element in $V_2^1 \otimes V_2^2$ is $\SL_2$, diagonally embedded into the last two factors of $L$. 
The stabilizer in $\SL_4$ of a generic element in $\wedge^2 V_4$ is $\Sp_4$. As in Section \ref{ss:A_3+A_1}, one can write 
$V_4 = V_2^3 \oplus V_2^4$, such that the stabilizer in $\Sp_4 \times \SL_2^1$ of a generic element in $V_4 \otimes V_2^1$ is $\SL_2^3 \times \SL_2$, where the second $\SL_2$ is diagonally embedded to $\SL_2^4$ and $\SL_2^1$.
Hence $S=\SL_2^3 \times \SL_2$, where the second $\SL_2$ is diagonally embedded in $\SL_2^4$, $\SL_2^1$ and $\SL_2^2$.
Let $\SL_{2,\mathfrak c}=\SL_2^3$.
Then $\frak{g}(1)=12 V_1 \oplus 5 V_2$, so $m=5$.

\subsection{$D_5(a_1) + A_2$}

\begin{align*}
L&=\SL_2^1  \times \SL_4\times \SL_2^2,\\
\frak{g}(1)&=V_2^1 \oplus  V_2^1 \otimes V_4^*  \oplus  \wedge^2 V_4\otimes V_2^2,\\
\frak{g}(2)&=V_1 \oplus V_4^* \oplus  V_2^1\otimes V_4 \otimes V_2^2  .
\end{align*}
As in  Section \ref{ss:2A_2+A_1},  the stabilizer of a generic element in $V_2^1 \otimes V_4 \otimes V_2^2$ is $\SL_2 \times \SL_2$, diagonally embedded in $L$, where 
embedding into $\SL_4$  is given by the tensor product. The stabilizer in $\SL_2 \times \SL_2$ of a generic element in $V_4^*$ is $\SL_2$, diagonally embedded. 
 Hence, $S$ is $\SL_2$ and,  as  $\SL_{2,\mathfrak c}$ module, $\frak{g}(1)= 3V_2 \oplus 2V_3 \oplus V_4$. Hence, only the method of M\oe glin works.

\subsection{$D_6(a_2)$}

\begin{align*}
L&=\SL_2^1 \times \SL_4 \times \SL_2^2,\\
\frak{g}(1)&=V_4 \oplus V_2^1 \otimes V_4^* \oplus V_4\otimes V_2^2 ,\\
\frak{g}(2)&=V_2^1 \oplus V_2^1 \otimes V_2^2 \oplus \wedge^2 V_4 \otimes V_2^2.
\end{align*}
The stabilizer in $\SL_2^1 \times \SL_2^2$ of a generic element in $V_2^1 \otimes V_2^2$ is $\SL_2$ diagonally embedded. The action of 
$\SL_4$ on $\wedge^2 V_4$ gives an isomorphism  $\SL_4/\mu_2 \cong \SO_6$. A generic element in $\wedge^2 V_4 \otimes V_2^2$ is contained in 
$V_2^3\otimes V_2^2$ where $V_2^3 \subset \wedge^2 V_4 $ is a non-degenerate 2-dimensional quadratic subspace. Write $\wedge^2 V_4= V_2^3\oplus V_4^3$. 
Hence the stabilizer in $\SO_4 \times \SL_2^2$ of a generic element in $\wedge^2 V_4 \otimes V_2^2$ is 
$\SO_2 \times \SO(V_4^3)$ where $\SO_2$ is diagonally embedded in $\SO(V_2^3)$ and $\SO_2\subset \SL_2^2$. 
Next, the stabilizer in $\SL_2^1 \times \SL_2^2$ of a generic element in $V_2^1 \otimes V_2^2$ is $\SL_2$ diagonally embedded. Furthermore, the stabilizer in 
$\SO_2$ of a generic point in $V_2^1$ is trivial. Hence the stabilizer $S$ of a generic point in $\frak{g}(2)$ is 
$\Spin(V_4^3)\cong \SL_2(K) \subset \SL_4$, where $K$ is a quadratic algebra. Let  $\SL_{2,\mathfrak c}=\SL_2(k)\subseteq \SL_2(K)$. Then the conditions 
(1)-(3) are satisfied with $m=10$. Since $10$ is even, Corollaries \ref{cor:local} and \ref{cor:global} do not apply i.e.  
 the Fourier-Jacobi model in  Propositions \ref{P5:local}  and \ref{P5:global}  is not a genuine representation of a 2-fold central extension of $\SL_{2, \mathfrak c}$. 
  However, as  $\SL_2(K)$-module, $\frak{g}(1)=5V_2^K$ 
where $V_2^K$ is the standard 2-dimensional representation over $K$. In particular, the Fourier-Jacobi model in  Propositions \ref{P5:local}  and \ref{P5:global}  is a genuine representation of 2-fold central extension of  $\SL_2(K)$ and  hence it is Whittaker generic as $\SL_{2,\mathfrak c}$-module. Thus the orbit can be raised using  Propositions \ref{P5:local}  and \ref{P5:global}.

\subsection{$E_6(a_3)+A_1$}

\begin{align*}
L&=\SL_2^1 \times \SL_2^2 \times \SL_4,\\
\frak{g}(1)&=V_2^1 \otimes V_2^2 \oplus V_4 \oplus V_2^2 \otimes \wedge^2 V_4,\\
\frak{g}(2)&=V_1 \oplus V_2^2 \otimes V_4^* \oplus V_2^1 \otimes\wedge^2 V_4.
\end{align*}
This is somewhat similar to the previous case, so let $\SO_2 \times \SO(V_4^3)$ be the stabilizer in $\SL_2^1 \times \SO_6$. Ignoring the 
factor $\SO_2$, we must compute the stabilizer in $\SL_2^2 \times \SL_2(K)$ of a generic point in  $V_2^2 \otimes V_4^*$. 
It is  $\SL_2$ diagonally embedded into $\SL_2^2 \times \SL_2(K)$. (This is evident when $K=k\oplus k$, and by Galois descent in general.)  
Hence $\SL_{2,\mathfrak c}$ is embedded into 
$\SL_2^2$ and $\SL_4$, so that $V_4=2V_2$, as $\SL_{2,\mathfrak c}$-module. Hence $\frak{g}(1)=8V_2 \oplus V_4$. Hence neither method works.

\subsection{$E_7(a_5)$}

\begin{align*}
L&=\SL_3^1 \times \SL_2^1 \times \SL_2^2 \times \SL_3^2,\\
\frak{g}(1)&=V_3^1\otimes V_2^1  \oplus V_2^1 \otimes V_2^2\otimes V_3^{2,*},\\
\frak{g}(2)&=V_3^2 \oplus V_3^1 \otimes V_2^2 \otimes  V_3^{2,*}.
\end{align*}
It is clear that $\SL_2^1$ is always in $S$. (The 18-dimensional summand in $\frak{g}(2)$ is the Bhargava $3\times 2 \times 3$ cube. 
A generic cube corresponds to a cubic separable algebra $E$ over $k$, and the stabilizer  in $\SL_3^1  \times \SL_2^2 \times \SL_3^2$ is 
isomorphic to $E^1 \rtimes \Aut^1(E)$ where $E^1$ is the torus of norm one elements in $E^{\times}$ and $\Aut^1(E)$ the group of 
$k$-automorphism of $E$ of determinant 1.) 
 Let $\SL_{2,\mathfrak c} = \SL_2^1$.
Then, as an $\SL_{2,\mathfrak c}$ module,
$\frak{g}(1)=9V_2^1$.
Hence, $m=9$.

\subsection{$D_5 + A_1$}

\begin{align*}
L&=\SL_2 \times \SL_4,\\
\frak{g}(1)&=V_2 \oplus V_4 \oplus V_2\otimes \wedge^2 V_4,\\
\frak{g}(2)&=2V_1 \oplus \wedge^2 V_4 \oplus V_2\otimes V_4^*.
\end{align*}
The same as  Section \ref{ss:A_3+A_1}.

\subsection{$D_6$}

\begin{align*}
L&=\SL_4,\\
\frak{g}(1)&=V_4^* \oplus 2 V_4,\\
\frak{g}(2)&=4V_1 \oplus \wedge^2 V_4.
\end{align*}
The same as Section \ref{ss:D_4+A_1}.

\subsection{$A_7$}

\begin{align*}
L&=\SL_2^1 \times \SL_2^2 \times \SL_2^3 \times \SL_2^4,\\
\frak{g}(1)&=V_2^1 \oplus V_2^2 \oplus V_2^4 \oplus V_2^2 \otimes V_2^3 \otimes V_2^4,\\
\frak{g}(2)&=V_1 \oplus V_2^1 \otimes V_2^2\oplus V_2^2 \otimes V_2^3 \oplus V_2^3 \otimes V_2^4 .
\end{align*}
The stabilizer $S$ of a generic point in $\frak{g}(2)$ is  $\SL_2$ embedded diagonally in the four $\SL_2$.  
Hence $\frak{g}(1)= 5V_2 \oplus V_4$, and the method of M\oe glin works. 

\subsection{$E_6 + A_1$}

\begin{align*}
L&=\SL_2^1 \times \SL_2^2 \times \SL_2^3,\\
\frak{g}(1)&=V_2^3 \oplus V_2^1 \oplus V_2^1 \otimes V_2^2 \otimes V_2^3,\\
\frak{g}(2)&=3V_1 \oplus V_2^1 \otimes V_2^2 \oplus V_2^2 \otimes V_2^3.
\end{align*}
The same as  Section \ref{ss:2A_2+A_1,E_6}.

\subsection{$E_7(a_2)$}

\begin{align*}
L&=\SL_2^1 \times \SL_2^2 \times \SL_2^3,\\
\frak{g}(1)&=V_2^2 \oplus V_2^1 \otimes V_2^2\oplus V_2^2 \otimes V_2^3 ,\\
\frak{g}(2)&=V_1 \oplus 2V_2^1 \oplus V_2^3 \oplus V_2^1 \otimes V_2^3.
\end{align*}

The same as  Section \ref{ss:D_6(a_2)}.

\subsection{$D_7$}

\begin{align*}
L&=\SL_2^1 \times \SL_2^2,\\
\frak{g}(1)&=2V_2^1 \oplus 3V_2^2,\\
\frak{g}(2)&=5V_1 \oplus V_2^1 \otimes V_2^2,\\
\frak{g}(3)&=2V_2^1 \oplus 3V_2^2,\\
\frak{g}(4)&=4V_1 \oplus V_2^1 \otimes V_2^2,\\
\frak{g}(5)&=2V_2^1 \oplus 2V_2^2.
\end{align*}
The stabilizer $S$ of a generic point in $\frak{g}(2)$ is $\SL_2$ embedded diagonally in the two $\SL_2$.
Hence $\frak{g}(1)=5V_2$, so $m=5$. Note that $\dim \frak{g}(0,2)=2$ and $\dim \frak{g}(2,2)=1$, and the $s_{\mathfrak c}$-weights are bounded by $2$.

\subsection{$E_7$}

\begin{align*}
L&=\SL_2,\\
\frak{g}(1)&=V_2 \oplus V_2 \oplus V_2,\\
\frak{g}(2)&=7V_1.
\end{align*}
Here $S=L$. Hence $\frak{g}(1) = 3V_2$, so $m=3$.

\section{Acknowledgment} 
The authors would like to thank Peter Trapa for a crystalizing conversation concerning the definition of metaplectic special orbits, Joseph Hundley for bringing attention to a paper of Monica Nevins and Monica Nevins for a correspondence on that paper. A part of this paper was written during 
the program on New Geometric Methods in Number Theory at MSRI, Berkeley. The authors have been partially supported by grants from NSF, DMS--1301567, 
DMS-1302122 and DMS-1359774, respectively. The second named author was also supported by a postdoctoral research fund from the Department of Mathematics, University of Utah.


\begin{thebibliography}{}

\bibitem[BS98]{BS98}  R. Berndt and R. Schmidt, {\it Elements of the representation theory of the Jacobi group.} Birkh\"auser, 1998.  


\bibitem[CM93] {CM93}
D. Collingwood and W. McGovern, 
{\it Nilpotent orbits in semisimple Lie
algebras.}
Van Nostrand Reinhold Mathematics Series. Van Nostrand
Reinhold Co., New York, 1993. xiv+186 pp.

\bibitem[Dyn52] {Dyn52}
E. B. Dynkin, 
{\it Semisimple subalgebras of semisimple Lie algebras.} (Russian) Mat. Sbornik N.S. \textbf{30} (72), (1952). 349--462

\bibitem[G06]{G06} D. Ginzburg, 
{\it Certain conjectures relating unipotent orbits to automorphic representations.} 
Israel J. Math. {\bf 151} (2006), 323-355. 

\bibitem[GRS03]{GRS03}
D. Ginzburg, S. Rallis and D. Soudry,
{\it On Fourier coefficients of automorphic forms of symplectic groups.}
Manuscripta Math. \textbf{111} (2003), no. 1, 1--16.


\bibitem[JN05]{JN} 
S. G. Jackson and A. G. N\"oel, 
{\em Prehomogeneous spaces associated with complex nilpotent orbits.} 
Journal of Algebra {\bf 289} (2005), 515--557. 

\bibitem[JL15]{JL15}
D. Jiang and B. Liu,
{\it On special unipotent orbits and Fourier coefficients for automorphic forms on symplectic groups}. J. Number Theory {\bf 146} (2015), 343--389. 

\bibitem[LS08]{LS} H. Y. Loke and G. Savin, {\em On minimal representations of Chevalley groups of type $\RD_n$, $\RE_n$ and $\RG_2$.} 
 Math. Ann. {\bf 340} (2008), no. 1, 195--208. 

\bibitem[Mo96]{Mo96}
C. M\oe glin, 
{\it Front d'onde des repr\'esentations des groupes classiques $p$-adiques.} 
American J. Math. {\bf 118} (1996), 1313--1346. 

\bibitem[MW87]{MW87}
C. M\oe glin and J.-P. Waldspurger, 
{\it Mod\`eles de Whittaker d\'eg\'en\'er\'es pour des groupes $p$-adiques}. (French)
Math. Z. {\bf 196} (1987), no. 3, 427--452. 


\bibitem[N02] {N02}
M. Nevins,
{\it Admissible nilpotent coadjoint orbits of real and  $p$-adic split  exceptional groups.}
Represent. Theory \textbf{6} (2002), 160--189.


\bibitem[Wal01]{W01}
J.-L. Waldspurger,
{\it Int\'egrales orbitales nilpotentes et endoscopie pour les groupes classiques non ramifi\'es}. Ast\'erisque \textbf{269}, 2001.

\bibitem[Wei03]{W03} 
M. Weissman, 
{\it The Fourier-Jacobi map and small representations}. Representation Theory {\bf 7} (2003), 275--299. 

\end{thebibliography}
\end{document}